\documentclass[11pt]{article}
\usepackage{amsfonts}
\usepackage{amssymb}
\usepackage{amsmath}
\usepackage{graphicx}
\usepackage{epsfig,epic,eepic}
\usepackage{color}
\usepackage[normalem]{ulem}
\usepackage{pinlabel}
\usepackage{amsxtra}
\usepackage{amsthm}
\DeclareMathOperator*{\id}{id}
\definecolor{gre}{rgb}{0.2,0.5,0.2}
\definecolor{purple}{rgb}{.5,0,.5}

\newcommand{\R}{\mathbb R}
\newcommand{\N}{\mathbb N}

\newcommand{\Z}{\mathbb Z}

\newcommand{\F}{\mathcal F}
\newcommand{\LL}{\mathcal L}
\newcommand{\arcosh}{\text{arcosh}}

\newcommand{\hop}{\vskip .2cm\noindent}

\newcommand{\zkkk}{\Z/4k\Z}
\newcommand{\opna}{\operatorname}
\newcommand{\zkk}{\Z/2k\Z}
\newcommand{\zk}{\Z/k\Z}

\newcommand{\e}{\varepsilon}

\newtheorem{thm}
{Theorem}[section]

\newtheorem{cor}[thm]{Corollary}
\newtheorem{prop}[thm]{Proposition}
\newtheorem{lem}[thm]{Lemma}
\newtheorem{defi}[thm]{Definition}
\newtheorem{rema}[thm]{Remark}
\newtheorem{fact}[thm]{Fact}
\newtheorem{quest}[thm]{Question}

\title{On spacelike Zoll surfaces with symmetries}
\author{Pierre Mounoud  \\
	Univ. Bordeaux, IMB, UMR 5251\\ 
	F-33400 Talence, France  \\
	{\it email: pierre.mounoud@math.u-bordeaux1.fr}  \\
	\and 
	Stefan Suhr  \\
	Fachbereich Mathematik, Universit\"at Hamburg\\ 
	Bundesstra\ss e 55\\ 
	20146 Hamburg, Germany \\
	{\it email: stefan.suhr@math.uni-hamburg.de} \\
	}
	\date{}
\begin{document}
\maketitle
\begin{abstract}
Three explicit families of spacelike Zoll surface admitting a Killing field are provided. It allows to prove the existence of spacelike Zoll surface not smoothly conformal to a cover of de Sitter space as well as the existence of Lorentzian M\"obius strips of non constant curvature all  of whose spacelike geodesics are closed. Further the conformality problem for spacelike Zoll cylinders is studied.
\end{abstract}

\section{Introduction}
A spacelike Zoll surface is a Lorentzian surface all of whose spacelike geodesics are simply closed and have the same length. The basic example of a spacelike Zoll surface is de Sitter space, the homogeneous space  $\operatorname{SO}_0(2,1)/\operatorname{SO}_0(1,1)$, and its finite coverings since it is not simply connected. It can be understood as the Lorentzian analogue of the round sphere as it has constant positive curvature. 
In \cite{MounoudSuhr}, the authors proved that a spacelike Zoll surface is diffeomorphic to a cylinder or a M\"obius strip.  This purely topological classification leaves open the finer questions of a classification up to isometry or conformality.
Recall that the cylinder as well as the M\"obius strip admit uncountable many non equivalent conformal Lorentzian structures.

The purpose of this article is  twofold. First it provides three infinite dimensional  families of examples of spacelike Zoll surfaces, in order to test answers to the questions that arise in the study of these surfaces. Second it tries to initiate a study 
of the conformal properties of spacelike Zoll surfaces since this is the main difference, besides the topological one, to the Riemannian case.

In the Riemannian case several explicit families of  Zoll surfaces, i.e.\ surfaces all of whose geodesics are simply closed, are known. The most famous family is certainly the Zoll spheres of revolution, i.e.\ with a Killing vector field, classified by Zoll and Darboux (see \cite{Besse} chap.\ 4).  The first work in this direction  for spacelike Zoll surfaces has been done by Boucetta \cite{Bou}, who provided examples of spacelike Zoll cylinders of revolution, i.e.\ admitting a periodic spacelike Killing field. However, contrary to a Riemannian 2-sphere, Killing fields of Lorentzian cylinders are not  periodic in general.
Already, on de Sitter space, there exist three conjugacy classes of Killing fields: 
the elliptic, the parabolic and the hyperbolic one, each corresponding to a conjugacy class of $1$-dimensional subgroups of $\operatorname {SO}_0(2,1)$ acting on $\operatorname{SO}_0(2,1)/\operatorname{SO}_0(1,1)$. We have thus decided to investigate spacelike Zoll surfaces admitting a non trivial Killing field.

For general spacelike Zoll cylinders the dynamics and the causal character of a Killing vector field coincides with that of a Killing vector field on de Sitter space (see Proposition \ref{prop_atlases}). Thus there exist only three types: elliptic, parabolic and hyperbolic. Adding technical assumptions allows  to  obtain three families of spacelike Zoll metrics, called elliptic, parabolic and hyperbolic according to  the dynamic of the Killing field, see Theorems \ref{theo_parabolic}, \ref{theo_elliptic} and \ref{theo_hyperbolic}. These metrics are constructed as deformations of a covering of de Sitter space,  the deformation preserving a chosen Killing field $K$. This corresponds to the construction of Zoll surfaces of revolution in chapter 4 of \cite{Besse}. In the case of an elliptic Killing vector field the family gives a complete classification. 
When $K$ has lightlike orbits, the deformations are realized via atlases adapted to $K$ (see Definitions \ref{def_para}, \ref{def_hyp}). These atlases are inspired by ideas used in \cite{BavardMounoud}.
 At the moment the authors are not aware of an example of a spacelike Zoll surface with a Killing vector field that does not belong to one of these families, but conjecture that such metrics exist.

 Besides the classification problem for Zoll metrics there is the rigidity problem for Zoll projective planes proven by Green \cite{Green} and
recently extended by Pries \cite{Pries} to surfaces all of whose geodesics are closed. These notes present a new feature of spacelike Zoll surfaces, opposing the Riemannian case with the existence of non constant curvature metrics on the M\"obius strip all of whose spacelike geodesics are closed. The examples constructed are covered by
 smooth spacelike Zoll metrics with non constant curvature and parabolic or hyperbolic Killing vector fields invariant by antipody, see Corollaries \ref{coro_parabolic} and \ref{coro_mob_hyp}. So far it is not clear whether these metrics are spacelike Zoll on the M\"obius strip. It is interesting however to note that none of the three families contain real-analytic metrics invariant by antipody.


Dropping the assumption of a Killing vector field  two major results on Riemannian Zoll surfaces remain: One is the theorem by Green \cite{Green} and its recent extension in \cite{Pries} mentioned before. The other one  is the theorem of  Guillemin \cite{Guillemin} saying that the space of Zoll metrics on $S^2$ in the conformal class of the constant curvature metric $g_0$ is a manifold near $g_0$ and the tangent space at $g_0$ is precisely the space of odd functions on $S^2$. If Guillemin confined his study to the conformal  deformations of the round sphere it is because of the uniformization theorem. Note that an uniformization theorem does not exist for Lorentzain surfaces and there exists an infinite number of non isometric conformal classes of Lorentzian cylinders. So naturally the question appears which conformal classes of Lorentzian cylinders are represented by spacelike Zoll metrics.   


This paper contains essentially three results on the conformal class of a spacelike Zoll cylinder $(C,g)$. Without further assumption,  it is shown that a two-fold cover of $(C,g)$ conformally embeds into de Sitter, see Proposition \ref{P3}. In the presence of a Killing field Theorem \ref{thm_conf_killing} shows that $(C,g)$ is always $C^0$-conformal to a cover of de Sitter space, i.e.\ that there exists a homeomorphism exchanging the lightlike foliations. The $C^0$-conformal class is determined by looking at the reflexion of  lightlike curves on the conformal boundary. 
Finally, the first family of example, the parabolic one 
contains  metrics that are not smoothly conformal (actually not $C^2$-conformal) to a cover of de Sitter space, see Theorem \ref{theo_nonsmoothconformal}. As often in Lorentzian geometry, the question of determining the conformal classes is quite subtle. The conformal class being simply given by a pair of foliations two metrics may be $C^n$ but not $C^{n+1}$-conformal for any $0\leq n\leq \infty$. The authors conjecture that there exist spacelike Zoll metrics with a Killing field which are not $C^1$-conformal to a cover of de Sitter. An extended classification of parabolic spacelike Zoll cylinders could
yield such a result.

The paper is organized as follows:  section \ref{sec_Zoll} studies spacelike Zoll surface without assuming the presence of a Killing field;  section \ref{section_general} gives the description of spacelike Zoll cylinders admitting a Killing field;  section \ref{sec_conf} determines the $C^0$-conformal class of these metrics; sections \ref{section_parabolic}, \ref{section_elliptic} and \ref{section_hyperbolic} are devoted to the construction of the families of examples, finally  section \ref{sec_Blaschke}, following an idea of Blaschke,  explains how it is possible to  blend the preceding constructions in order to find examples that do not admit any Killing fields.

{\bf Acknowledgement} The first author wishes to thank J.F.\ Bony  for useful discussions and J.\ Lafontaine for pointing at him the question of the existence of  non constant curvature spacelike Zoll M\"obius strip.  The second author wishes to thank J.-L.\ Flores for explanations in connection with the causal boundary and conformal boundary of spacetimes.




\section{ General spacelike Zoll surfaces}\label{sec_Zoll}

\begin{prop}\label{P1}
Let $(C,g)$ be pseudo-Riemannian 
cylinder all of whose spacelike geodesics are closed. Then $(C,g)$ is globally hyperbolic 
and the universal cover of $(C,-g)$ is not globally hyperbolic.
\end{prop}

\begin{lem}\label{lem_orient}
Let $(C,g)$ be a Lorentzian cylinder with at least one non timelike or non spacelike loop. Then $(C,g)$ is space- and time-orientable. 
\end{lem}

\begin{proof}
By exchanging $g$ with $-g$, if necessary, we can assume that the loop in the assumption is non timelike.
Further since $C$ is a surface we can assume that the loop is simply closed. Well-known
arguments in Lorentzian geometry (cp. \cite{penrose}) allow us to additionally assume that the loop is smooth and
regular.

Let $\gamma\colon[0,1]\to C$ be a simple closed, smooth and regular non timelike loop in $(C,g)$. Further let $v_l$ and $v_r$ be lightlike vectors at 
$\gamma(0)$ pointing to the same side of $\gamma$. If $\gamma$ itself is a lightlike pregeodesic, the subsequent argument will apply to the lightlike
direction not tangent to $\gamma$. Denote with $\eta_l$ and $\eta_r$ the geodesics with direction $v_l$ and $v_r$. Lift both to $\R^2$. Then not both lightlike 
geodesics can be invariant (even up to a finite quotient) under the group of deck transformations. So w.lo.g. we can assume that the lift of $\gamma$ lies on one 
side of $\widetilde\eta_l$ the lift of $\eta_l$. Now consider the strip bounded by $\widetilde\eta_l$ and its translate by the deck transformation $\alpha$ induced
by the fundamental class of $\gamma$. If $\dot\gamma(1)$ does not lie in the same connected
component of $\{w\in TC_{\gamma(0)}|\; g(w,w)\ge 0\}\smallsetminus \{0\}$ as $\dot\gamma(0)$, then
$\widetilde\gamma$ and $\alpha\circ \widetilde\gamma$ lie in the strip bounded by
$\widetilde\eta_l$ and $\alpha\circ\widetilde\eta_l$. But then $\alpha$ will have a fixed point in that
strip, which contradicts the assumption that $\alpha$ is a deck transformation.

Consequently $(C,g)$ is space-orientable. Together with the orientability of $C$ this implies the
time-orientability of $(C,g)$ as well.
\end{proof}

\begin{proof}[Proof of Proposition \ref{P1}.]
First we prove that $(C,g)$ is globally hyperbolic. W.l.o.g. we can assume that $(C,g)$ is spacelike Zoll. It is well known that global 
hyperbolicity is passed down to finite quotients. Let $\gamma$ be any spacelike geodesic of $(C,g)$. By assumption $\gamma$ is 
an embedded closed hypersurface. We claim that $\gamma$ is a Cauchy hypersurface in $(C,g)$. Let $\eta$ be an inextendable causal curve in $(C,g)$ 
that does not intersect $\gamma$. Choose any curve from a point on $\gamma$ to a point on $\eta$ and parallel transport the tangent vector $\dot\gamma$ 
along that curve. Denote the transported vector by $v$. Since $\eta$ is causal the spacelike geodesic with direction $v$ is transversal to $\eta$ 
(w.l.o.g. we can assume $\eta$ to be smooth.). This induces a smooth family of closed curves transversal to a given curve at one end and disjoint 
at the other end. This is of course impossible. Therefore $\eta$ intersects $\gamma$ and $\gamma$ is a Cauchy hypersurface. 

 Next we show that the universal cover $(\widetilde{C},-\widetilde{g})$ of $(C,-g)$ is not globally hyperbolic. Consider a deck transformation $\phi$ of $\widetilde{C}\to C$ and a 
point $p\in \widetilde{C}$. W.l.o.g. we can assume that $\phi(p)\in I^+(p)$,  relative to the time-orientable metric $\widetilde{g}$. Else consider $\phi^{-1}$ instead of $\phi$. Now if $(\widetilde{C},
-\widetilde{g})$ is globally hyperbolic then the space of causal arcs between $p$ and $\phi(p)$ is compact. By the limit curve lemma and the assumption 
that every future pointing $(-\widetilde{g})$-timelike geodesic from $p$ intersects $\phi(p)$ we can conclude that both lightlike geodesics emanating from $p$ intersect 
$\phi(p)$. Since they are curves belonging to different transversal foliations this is impossible.
\end{proof}

\begin{prop}\label{P2}
Let $(C,g)$ be a  Lorentzian cylinder all of whose spacelike geodesics are closed. Then any pair of spacelike 
geodesics intersects at least twice. The number of intersections is even and constant throughout the set of spacelike geodesics.
\end{prop}
\begin{proof}
The second assertion follows from the first, since obviously there are intersecting spacelike geodesics and any pair of loops in the cylinder has to intersect at least
twice if they intersect once. 

Since $(C,g)$ is spacelike Zoll we can assume with \cite{MounoudSuhr} that the geodesic flow on $T^1C$ is a free $S^1$-action. Fix a unit speed simply closed 
spacelike geodesic $\gamma$ and consider the tangent curve $\dot\gamma$. Further let $\eta_1,\eta_2$ be two unit speed simply closed spacelike geodesics 
geometrically different from $\gamma$. Then the tangent curves $\dot\eta_1$ and $\dot\eta_2$ are disjoint from $\dot\gamma$. Since $\dot\gamma$ is a loop in 
the $3$-manifold $T^1C$ we can connect $\dot\eta_1(1)$ and $\dot\eta_2(1)$ via  a path $\mu\colon [0,1]\to T^1C$ not intersecting $\dot\gamma$. The geodesics 
with initial direction $\mu(s)$ form a (smooth) homotopy by spacelike geodesics geometrically different from $\gamma$  with endpoints $\eta_1$ and $\eta_2$. Since any 
intersection between geometrically different geodesics is transversal, the number of intersection between $\gamma$ and $\eta_1$ has to coincide with the number 
of intersection between $\gamma$ and $\eta_2$. 
\end{proof}

\begin{lem} 
Every globally hyperbolic $2$-dimensional cylinder is globally conformally flat. 
\end{lem}

This fact is actually well known. For the sake of completeness we give a proof. Note that the lemma shows that every globally hyperbolic cylinder 
is conformal to one connected component of the complement of at most two simply connected and disjoint non timelike curves. Conversely of course 
every such component is globally hyperbolic.
\\
\begin{proof}
Denote the lightlike foliations of $(C,g)$ by 
$\mathcal{F}_{1,2}$. Let $\gamma$ be any smooth Cauchy hypersurface in a globally hyperbolic cylinder $(C,g)$. Choose a diffeomorphism 
$\varphi\colon \gamma\to S^1$. Define two maps $\alpha,\beta \colon C\to S^1$ to be identical to $\varphi$ on $\gamma$, $\alpha$ to be constant 
on the leafs of $\mathcal{F}_1$ and $\beta$ to be constant on the leafs of $\mathcal{F}_2$. Since the lightlike foliations are transversal
the differentials of $\alpha$ and $\beta$ are linearly independent at every point. Lifting everything to the universal cover gives two coordinates
$x,y$ whose level sets are lightlike. Therefore the metric in these coordinates reads $f(x,y)dxdy$ with $f(x+2\pi,y+2\pi)=f(x,y)$. 
Consequently $f$ descends to the quotient and the metric $\frac{1}{f}g$ is flat.
\end{proof}

\begin{rema}
Next we want to fix a conformal embedding of de Sitter space into $(S^1\times\R,d\varphi^2-dt^2)$.  DeSitter space is isometric to 
$(S^1\times \R,\cosh^2(t)d\varphi^2-dt^2)$. So in order to construct a conformal embedding into the flat cylinder we have to find a 
reparameterization $\psi \colon (0,b)\to \R$ such that $(\id\times \psi)^*(\cosh^2(t)d\varphi^2-dt^2)$ is diagonal. This is equivalent to solving the ODE
$(\psi')^2(s)=\cosh^2\psi(s)$. Since $\psi$ is supposed to be a diffeomorphism we can assume that $\psi'>0$. Therefore we have to solve the equation
$\psi'(s)=\cosh\psi(s)$. In fact we do not need the solution $\psi$ explicitly. All we require is the value $b$, i.e.\ the length of the domain of $\psi$. 
This can be done by integration: We know that 
$$\int_0^s \frac{\psi'(\sigma)}{\cosh\psi(\sigma)}d\sigma=s$$ 
for all $s\in (0,b)$. For $t=\psi(s)$ we then see that (wlog $\lim_{s\to 0}\psi(s)=-\infty$)
$$\int_{-\infty}^t \frac{1}{\cosh(\tau)}d\tau=\psi^{-1}(t).$$ 
The left hand side is equal to $2\arctan e^t$ (which give the solution $\psi=\tan(\log \frac{s}{2})$) and therefore tends to $\pi$ for $t\to\infty$. Thus de Sitter space is conformal to a flat cylinder with circumference
$2\pi$ and height $\pi$. 
\end{rema}

\begin{prop}\label{P3}
Let $(C^2,g)$ be Lorentzian spacelike Zoll cylinder. Then for all $\e>0$ there exists a smooth conformal embedding of $(C,g)$ 
into $(S^1\times \R,\cosh^2(t)d\varphi^2-dt^2)$ whose image is contained in $S^1\times (-\e,\pi+\e)$. 
Especially up to a twofold covering $(C,g)$ admits a conformal embedding into de Sitter space. 
\end{prop}

 Note that the conformal embedding is not surjective in general. Finite coverings of de Sitter serve as examples.

\begin{proof}
Let $(C,g)$ be a spacelike Zoll surface and $F\colon C\to S^1\times\R$ be a conformal embedding. Consider the image of a 
lightlike geodesic of $(C,g)$ in $S^1\times\R$. We can assume that $(0,0)$ lies on the image and that the image is symmetric about $(0,0)$. 
Denote with $(a,a)$  and $(-a,-a)$ its future and past endpoint respectively on $\partial F(C)$ in $S^1\times \R$. Since $(C,g)$ is spacelike Zoll the fundamental class of every
spacelike geodesic generates $\pi_1(C)$ and via the conformal embedding their images generate $\pi_1(S^1\times\R)$. Note that due to the 
continuity of the geodesic flow for every pair of neighborhoods of $(a,a)$ and $(-a,-a)$ there exist spacelike geodesics passing through $(0,0)$ 
and intersecting these neighborhoods.  The geodesics have to close up after one round. Since the circumference of the circle is $2\pi$, the 
value of $4a$ is bounded from above by $2\pi$, i.e.\ $a\le \pi/2$.

Let $\gamma$ be a lightlike geodesic in $(C,g)$. Denote with $x$ and $y$  the future resp. past endpoints of the $F\circ \gamma$ in $S^1\times \R$. Since 
$(C,g)$ is globally hyperbolic the boundary of $F(C)$ is achronal (see e.g. \cite{FHS} Theorem 3.29 and 4.16). Therefore we have $F(C)\subseteq S^1\times\R \smallsetminus (I^+(x)\cup I^-(y))$. The last set being 
compact shows that $\partial F(C)$ consists of two simply closed disjoint non timelike curves $\gamma^\pm \colon S^1\to S^1\times\R$  with $t\circ\gamma^+>t\circ\gamma^-$. Since neither 
$\gamma^+$ nor $\gamma^-$ can be everywhere lightlike, we can approximate both curves  up to a given error $\e>0$ by smooth simply closed disjoint spacelike curves 
$\gamma^\pm_\e\colon S^1\to S^1\times\R$. Note that the precompact component of $S^1\times\R \smallsetminus (\gamma^+_\e(S^1)\cup \gamma^-_\e(S^1))$ is a 
globally hyperbolic spacetime which is $\e$-close to $F(C)$ whenever $\gamma^\pm_\e$ are $\e$-close to $\gamma^\pm$. 

Now consider the cylinder $(S^1\times\R,d\varphi^2-dt^2)$ as the quotient of $(\R^2,dxdy)$ by the $\Z$-operation generated by $(x,y)\mapsto (x+\sqrt{2}\pi,y
+\sqrt{2}\pi)$. Denote with $\widetilde\gamma^\pm_\e$ the lift of $\gamma^\pm_\e$ to $\R^2$. W.l.o.g. we can assume that $\widetilde\gamma^\pm_\e$ are 
parameterized as graphs over the $x$-axis, i.e.\ $\widetilde\gamma^\pm_\e(s)=(s,\theta^\pm_\e(s))$ for some maps $\theta^\pm_\e\colon\R\to\R$. Since 
$\gamma^\pm_\e$ are spacelike and simply closed $\theta^\pm_\e$ are $\Z$-equivariant diffeomorphisms. 

Set $\theta_\e(s):=\frac{1}{2}(\theta^+_\e(s)+\theta^-_\e(s))$. $\theta_\e$ is obviously a diffeomorphism of the reals. Define the diffeomorphism 
$\Theta\colon \R^2\to\R^2$ by $\Theta(x,y):=(x,\theta_{\e}^{-1}(y))$. It is conformal and maps the 
spacelike curve $s\mapsto (s,\theta_\e(s))$ to the diagonal $\triangle =\{(x,y)|\; x=y\}$. We know that 
$$\sup_s|\theta_{\e}^+(s)-\theta_{\e}^-(s)|\le \sqrt{2}(\pi+\e),$$
i.e.\ therefore the curves $\Theta\circ \widetilde\gamma^\pm_\e$ have distance at most $\frac{\pi+\e}{2}$ from the diagonal. This implies the same maximal 
distance from the sets $\{t=0\}$ in the quotient space. Since for the chosen conformal embedding of de Sitter we have $\theta^\pm(s)=s\pm\frac{\pi}{\sqrt{2}}$, 
the claim follows. 
\end{proof}


\section{ Killing fields on spacelike Zoll surfaces.}\label{section_general}
From now on, we will be interested in spacelike Zoll cylinders admitting a Killing field $K$. We will prove in this section that the 
dynamics of $K$ are always similar to that of a Killing field of a cover of de Sitter space. See Proposition \ref{prop_atlases} for the precise statement.
\begin{prop}
Let $(S,g)$ be a  connected Lorentzian surface all of whose spacelike geodesics are closed. Then any locally Killing vector field of $(S,g)$ is complete and is therefore Killing. 
\end{prop}
\begin{proof} Let $K$ be locally Killing vector field on $S$ and let $\Phi_K$ be its local flow. For any $z\in S$,  we define $\omega_z$ by  $\omega_z=\sup\{t;\Phi_K^t(z)\ \text{exists}\}$. 
Let $\gamma$ be  a spacelike geodesic and $\omega_\gamma=\inf_{z\in \gamma} \omega_z$. Let us assume that there exists $x\in \gamma$ such that $\omega_x>\omega_\gamma$ then 
$\Phi^K_{\omega_\gamma}(\gamma)$ is a  spacelike geodesic that is not contained in any compact subset of $S$. This clearly contradicts the assumption 
that all spacelike geodesics are closed. Therefore  we have $\omega_x=\omega_\gamma$ for any $x\in \gamma$. Since any pair of points in $S$ can be joined by a broken spacelike geodesic, 
the function $\omega$ is constant and  the flow of $K$  is complete.\end{proof}

%
\begin{prop}\label{prop_cases}
 Let $(C,g)$ be a spacelike Zoll Lorentzian cylinder admitting a non trivial Killing field $K$. Then every  spacelike geodesic of $g$ is at least twice tangent to $K$ or contains at least two zeros of $K$. It follows that  
 \begin{enumerate}
  \item $K$ is   periodic if and only if it is  spacelike and if and only if it has a recurrent orbit;
  \item $K$ is vanishing if and only if $K$ is somewhere timelike;
  \item any geodesic  perpendicular to $K$ contains all its zeros.
 \end{enumerate}
In particular, $K$ has to be spacelike somewhere. 
 Further $K$ has only finitely many lightlike orbits.
\end{prop}

\begin{proof} Let $\gamma$ be a spacelike geodesic of $g$. If there exists $t_0$ such that $\gamma\smallsetminus \gamma(t_0)$ is transverse to $K$ then by pushing $\gamma$ along the flow of $K$ gives a spacelike geodesic intersecting $\gamma$ 
in at most one point. But, according to Proposition \ref{P2}  this is impossible.

It is well known that if $K$ vanishes then it is somewhere timelike. Reciprocally, if $K$ is  timelike at a point $x\in C$, consider the  geodesic $\gamma$ defined by $\gamma(0)=x$, $g(\gamma'(0),\gamma'(0))=1$ and  $g(\gamma'(0),K_x)=0$. 
As $\gamma$ cannot be tangent to $K$ and as it cannot be everywhere transverse to it, $K$ vanishes somewhere along $\gamma$.  Moreover, if $m$ is a zero of $K$ and  $\gamma'$ a spacelike geodesic containing $m$, we choose $\gamma'$ 
different from $\gamma$. By Proposition \ref{P2} $\gamma'$ intersects $\gamma$. As $\gamma$ and $\gamma'$ cannot be tangent and as they are both perpendicular to $K$,  the intersection can be only 
at a zero of $K$. It follows that there exists $t_0$ such that $\Phi^{t_0}_K(\gamma)=\gamma'$ and therefore $m\in \gamma$ and any zero of $K$ is on $\gamma$.

According to \cite{MounoudSuhr}, a spacelike  Zoll surface has no closed lightlike geodesic, therefore a periodic Killing field has to be spacelike.
Reciprocally, if $K$ is a spacelike Killing field and if $\gamma$ is a spacelike geodesic that is not an orbit of $K$ then Clairaut's first integral imposes the value of $g(K,K)$ at the points where $K$ and $\gamma$ are tangent. 
It follows that $K$ is tangent to $\gamma$ only at points where the restriction of the function $\alpha$ defined by $\alpha(x)=g_x(K_x,K_x)$ to $\gamma$ reaches its maximum. But as $\gamma$ is compact this function also has a minimum, 
a point $x_0$ realizing this minimum has to be a critical point of $\alpha$, considered as a function on $C$.  The orbit of $K$ through $x_0$ is therefore  a spacelike geodesic and so is closed.  

If $K$ has a recurrent orbit, there exists a spacelike geodesic $\gamma$ intersecting transversally this orbit at a point $x$ and $t_0>0$ such $\Phi_K^{t_0}(x)\in \gamma$. Since a geodesic $\eta$ emanating from $\Phi_K^{t_0}(x)$ 
is uniquely determined by $g(\dot \eta,\dot \eta)$ and $g(K,\dot \eta)$, we have $\Phi^{t_0}_K(\gamma)=\gamma$. Consequently, $\Phi_K^{t_0}$ is an isometry of $\gamma$ seen as a Riemannian circle. Therefore the orbit of $\Phi_K^{t_0}$ 
is dense or finite.  It cannot be  dense as $K$ would be everywhere transverse to $\gamma$. Hence, it is finite and $K$ has a closed orbit. 

The proposition follows then from the fact that a (complete) Killing field on Lorentzian surface that has a closed leaf is periodic. Indeed, every geodesic emanating from a point $x$ contained in a closed orbit of $K$ is mapped 
to itself by $\Phi_K^{t_0}$ for some $t_0>0$. The isometry $\Phi_K^{t_0}$ has a fixed point $x$  and its differential $\opna{d}\Phi^{t_0}_K(x)$ is an element of $\opna{SO_0}(1,1)$ (if $C$ is not assumed to be orientable we 
replace $t_0$ by $2t_0$) having an eigenvalue equal to $1$ (associated to the eigenvector $K_x$) therefore  $\opna{d}\Phi_K^{t_0}=\opna{Id}$ and therefore $\Phi_K^{t_0}=\opna{Id}$.

Assume that $K$ has an infinite number of lightlike orbits. Since spacelike Zoll surfaces are globally hyperbolic, any given spacelike geodesic $\gamma$ intersects all lightlike 
geodesics that contains lightlike orbits of $K$. If the complement of the lightlike orbits has only finitely many connected components, then an open subset of $C$ is foliated by 
lightlike orbits of $K$. If not, the complement has infinitely many connected components. Since $K$ is smooth and $\gamma$ is compact, there exist an infinite number of these 
components on which $K$ is transversal to $\gamma$. Choose any such connected component. If $K$ is spacelike on it, then $g(K,K)$ has a maximum on the intersection of $\gamma$ 
with that component. But then one orbit of $K$ is a spacelike geodesic and therefore closed. This contradicts the first part of the proof. So these infinitely many connected 
components $K$ has to be timelike. We can therefore choose $\gamma$ to be orthogonal to $K$. But it means that $\gamma$ cuts only timelike or singular integral curves of $K$, 
therefore $\gamma$ contains an infinite number of $0$ of $K$. Which is impossible since $\gamma$ is compact. Therefore we can assume that an open subset of $C$ is foliated by 
lightlike orbits of $K$. It follows from the Lorentzian version of Wadsley's theorem, see \cite[Theorem 2.3]{MounoudSuhr}, that the set of  lengths of spacelike geodesics of a 
spacelike Zoll metric is bounded. On the other side, if $S$ is a strip foliated by lightlike orbits of $K$ then $S$ is flat and isometric to $(I\times \R, dxdy)$ for some interval 
$I$.  Thus, for any $T>0$ there exists a spacelike geodesic segment contained in $S$ whose length is greater than $T$. Hence, $(C,g)$ does not contain any  strip  foliated by lightlike orbits of $K$.
\end{proof}

\begin{prop}\label{prop_llgeod}
Let $(C,g)$ be a spacelike Zoll cylinder  admitting a non trivial Killing field $K$. Let $\eta$ be a lightlike geodesic of $g$ that is transverse to $K$ and $\alpha$ be 
the function on $C$ defined by $\alpha(x)=g_x(K_x,K_x)$. Then the function $\alpha$ tends to $+\infty$ at both ends of $\eta$. Moreover, if $K$ is not periodic 
then $\alpha$  vanishes once or twice on  $\eta$ and if there exists $x\in \eta$ such that $\alpha(x)<0$ then it vanishes exactly twice.
\end{prop}

\begin{proof} Let $\eta\colon (t_{\inf{}},t_{\sup{}})\to C$ be a lightlike geodesic that contains a point $x$ such that $\alpha(x)>0$.  For any $t$, if $\alpha(\gamma(t))>0$ then $(\alpha\circ \gamma)'(t)\neq 0$ or 
$(\alpha\circ \gamma)''(t)>0$. Indeed, if  $(\alpha\circ \gamma)'(t)=0$ then the orbit of $K$ through $\gamma(t)$ is a geodesic and therefore, as $g$ is spacelike Zoll,  contains conjugate points. The curvature of $g$ 
being constant along this geodesic, it has to be positive and therefore $(\alpha\circ \gamma)''(t)>0$ (see Lemma 4.9 of \cite{Besse}). 

Clairaut's first integral tells us that along any spacelike geodesic $\gamma$  the value of $g(\dot \gamma,K)$ is constant. We denote it by $k_\gamma$.  Even if 
we impose $g(\dot \gamma,\dot \gamma)=1$, it  can be chosen as big as wanted by taking an initial speed at a point where $\alpha>0$ sufficiently close to a 
lightlike direction.  Moreover  if $K$ is tangent to $\gamma$  and if $g(\dot \gamma,\dot \gamma)=1$ then the value of $\alpha$ at this point is  equal to 
$k_\gamma^2$. Any spacelike geodesic being somewhere tangent to $K$, the function $\alpha$ is unbounded on $C$.

We suppose first that $K$ is periodic. The saturation of any lightlike geodesic $\eta$ by $K$ is equal to $C$. We have seen  in  the proof of Proposition 
\ref{prop_cases} that $\alpha$ has critical points. But, we just saw they are all local minima. It means that $\alpha$ has a minimum that is realized on a unique 
orbit of $K$ that we denote by $\gamma_0$. We choose $\eta(0)$ such that $\eta(0)\in \gamma_0$, i.e.\ such that it realizes the minimum of 
$\alpha$. The restriction of $\alpha\circ \eta$ to $]-\infty,0[$ and $]0,\infty[$ are strictly monotonous and, as any spacelike geodesic has to be tangent to $K$ on 
both side of $\gamma_0$, we see that  $\alpha\circ\gamma\rightarrow +\infty$  when $t$ goes to  $t_{\inf{}}$ or $t_{\sup{}}$.

We can assume now that $\alpha$ vanishes  somewhere (but maybe not $K$). Let $\eta$ be a lightlike geodesic that is transverse to $K$. 
Let us first suppose that $\alpha(x)>0$ for some $x\in \eta$. Let $V$ be the connected component of $\alpha^{-1}(]0,\infty[)$ that contains $x$. The vector field 
$K$ sends lightlike geodesics to lightlike geodesics and leaves $V$ invariant. On $V$ the vector field $K$ is transverse to any lightlike 
geodesic, thus the flow of $K$ defines  an open equivalence relation on the lightlike geodesics of $V$. Hence, $V$ is the saturation of $\eta\cap V$ by the flow of 
$K$.  Let $\gamma$ be a spacelike geodesic intersecting $V$. As the function $\alpha$ vanishes on $\partial V$ the restriction of $\alpha$ to $\gamma\cap V$ 
has a local maximum and therefore $\gamma$ has to be tangent to $K$ somewhere in $V$. Therefore $\alpha$ is unbounded on $V$. Any level set of 
$\alpha|_V$ intersects $\eta$  since the saturation of $\eta\cap V$ under $K$ is $V$. Since $\alpha|_\eta$ is monotonous, the restriction of $\alpha$ to 
$\eta\cap V$ goes to $+\infty$ at one end and to $0$ at the other end. 

Let us see now that this other end corresponds to an intersection between $\eta$ and a lightlike orbit of $K$. 
The function $\alpha$ is strictly monotonous on $\eta\cap V$, therefore for any $c>0$, $\alpha^{-1}(c)\cap V$ is equal to one orbit of $K$. Any spacelike geodesic is twice tangent to $K$, therefore the boundary of $V$ contains at least two non trivial lightlike orbits of $K$. 
We  choose  small lightlike transversal $\tau_1$ and $\tau_2$ along each of them such that $\alpha(\tau_1\cap V)=\alpha(\tau_2\cap V)$. As the level sets of $\alpha$ on $V$ are equal to orbits of $K$, there exists $t_1$ such that $\Phi_K^{t_1}(\tau_1)\cap \tau_2\neq \emptyset$. If $\tau_1$ and $\tau_2$ are pieces of leaves from the same lightlike foliation then $\Phi_K^{t_1}(\tau_1)\cap V=\tau_2\cap V$. But it would mean that $\tau_1$ and $\tau_2$ are transversals of the same lightlike orbit of $K$, contrary to our assumption. Hence we can assume that $\eta$ and the geodesic containing $\tau_1$ are leaves of the same foliation. Consequently there exists $t_2$ such that $\Phi_K^{t_2}(\tau_1)\cap V\subset \eta$ and therefore $\eta$ intersects a lightlike orbit of $K$.


In order to see that $\alpha$ goes to $+\infty$ we just have to prove that it takes positive values again. Let $x$ be a point of $\eta$ such that $\alpha(x)=0$. According to Proposition \ref{prop_cases}, the function $\alpha$ has to take non zero values again.
Let us suppose that there exists a point  $y\in \eta$ such that $\alpha(y)<0$. We choose a parametrization of $\eta$ starting from $y$ and  $3$ unit spacelike geodesics $\gamma_i$ starting also from $y$. The geodesic $\gamma_0$ is perpendicular to $K$ and the initial speeds satisfy 
$$
|g(\dot \gamma_1(0),\dot \eta(0))|<|g(\dot \gamma_0(0),\dot \eta(0))|<|g(\dot \gamma_2(0),\dot \eta(0))|.
$$
That is $\gamma_1$ is the closest to $\eta$. We remark that the roles of $\gamma_1$ and $\gamma_2$ are permuted if $\eta$ is replaced by $\eta'$ the other lightlike geodesic emanating from $y$.

As above we choose two numbers $a<0<b$ such that $\gamma_2$ is transverse to $K$ on $]a,b[$ and  
such that $K$ is tangent to $\gamma_2$ at the points $\gamma_2(a)$ and $\gamma_2(b)$. 
Let $U$ be the saturation by $K$ of $\gamma_2(]a,b[)$.
Let us see that each orbit of $K$ cuts at most once $\gamma_2(]a,b[)$. If $y\in \gamma_2(]a,b[)$ and $\Phi_K^{t_0}(y)\in \gamma_2(]a,b[)$ then, using Clairaut's first integral and the fact that a flow always preserves the orientation, we see that $\Phi_K^{t_0}(\gamma_2)=\gamma_2$. As the set of tangency points between $\gamma$ and $K$ is preserved by the flow of $K$ it follows that $\Phi_K^{t_0}(\gamma_2(]a,b[))=\gamma_2(]a,b[)$. 
If $t_0\neq 0$ then $K$ has closed orbits contrarily to our assumption. The map $(s,t)\mapsto \Phi_K^s(\gamma_2(t))$ therefore defines coordinates on $U$ such that metric reads $\alpha(t)ds^2+2dsdt+dt^2$ (in order to obtain a $2$ we may have to change $K$ by one of its multiples), 
with  $\alpha(0)<0$.
The open set $U$ contains  lightlike orbits of $K$. Let $c\in ]0,b[$ the smallest number such that  $\alpha(c)=0$. It corresponds to an orbit of $K$ that goes to a zero of $K$ (a separatrix). It implies that $\alpha'(c)>0$ (otherwise $D_KK(\gamma_2(c))=0$) therefore $\alpha'(t)>0$ for $t\geq c$. Doing the same for the biggest number $d\in]a,0[$ such that $\alpha(d)=0$, we see that $U$ contains exactly $2$ lightlike orbits of $K$ that are separatrices of saddle points. 

The intersection of $\gamma_0$ with $U$ is asymptotic to these lines therefore $\eta\cap U$ cannot cut any of them and it has also to be asymptotic to them in both direction. It implies that $\eta'$ cuts the two lightlike orbits of $K$ contained in $U$ (see figure \ref{fig_gammas}). Swapping the roles of $\eta$ and $\eta'$ we see that $\eta$  cuts the two lightlike orbits of $K$ contained in the open set $U'$ obtained by saturating a segment of $\gamma_1$. 

\begin{figure}[h!]
\labellist
\small\hair 2pt
\pinlabel {$\eta'$} at 350 230
\pinlabel {$\gamma_2$} at  405 230
\pinlabel  {$\gamma_0$} at 612 187
\pinlabel  {$\eta$} at 610 152
\endlabellist
 \centering
 \includegraphics[scale=.5]{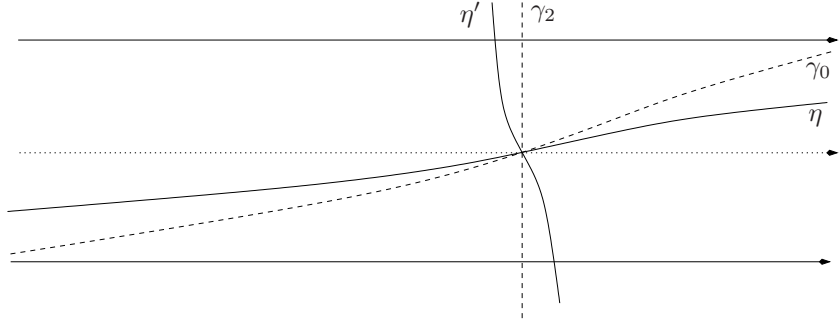}
 \caption{the positions of the curves  $\gamma_0$, $\gamma_2$, $\eta$ and $\eta'$ on $U$.}\label{fig_gammas}
\end{figure}

Thus there are points on $\eta$ on both side of $y$ where $\alpha$ takes positive values therefore $\alpha$ goes to infinity on both ends of $\eta$. \end{proof}

If $(C,g)$ is a spacelike Zoll surface with a Killing field $K$ then 
Proposition \ref{prop_llgeod} says that in the coordinates obtained by $K$-saturation of a lightlike geodesic the metric reads 
$h(y)dx^2+2dxdy$ with $h$ defined on $\R$ and $h\rightarrow +\infty$ when $y\rightarrow \pm \infty$. We can actually precise this fact:

\begin{prop}\label{prop_atlases}
 Let $(C,g)$ be a spacelike Zoll cylinder admitting a non zero Killing field $K$. 
 \begin{enumerate}
  \item when $K$ is periodic, then $(C,g)$ is the quotient by an horizontal translation, of a metric on $\R^2$ that reads $h(y)dx^2+2dxdy$ where 
  $h$ is  a positive function that has  a unique local minimum and verifies $\lim_{y\rightarrow\pm \infty}h(y)=+\infty$.
  \item when $K$ does not vanish and is not periodic, then there exists a finite atlas $\{(U_i,\psi_i), i\in \Z/2k\Z\}$ such that $\psi_i(U_i)=\R\times I_i$ and $\psi_i^{-1}{}^*g=h_i(y)dx^2+2dxdy$ where the $h_i$ are  non negative smooth functions, such that
  \begin{itemize}
  \item $h_i(0)=0$ 
  \item the $h_i$ are  strictly monotonous on $]-\infty,0[$ and on $]0,+\infty[$ 
  \item $\lim_{y\rightarrow \pm\infty} h_i(y)=+\infty$;
  \item $h_{2i}(t)=h_{2i+1}(t)$ for any $t> 0$ and $h_{2i}(t)=h_{2i-1}(t)$ for any $t< 0$;
  \end{itemize}
  \item when $K$ vanishes, then there exists a finite atlas $\{(U_i,\psi_i), i\in \Z/4k\Z\}$ of $C$ minus the set of zeros of $K$  such that $\psi_i(U_i)=\R\times I_i$ and $(\psi_i^{-1})^*g=h_i(y)dx^2+2dxdy$ where the $h_i$ are  smooth functions  such that:
  \begin{itemize}
   \item for any $i\in A$, there exists $a_i<0<b_i$ satisfying $h_i(a_i)=h_i(b_i)=0$, 
  \item $h_i$ is positive and strictly monotonous on $]-\infty,a_i[$ and on $]b_i,+\infty[$,
  \item $h_i$ is negative on $]a_i,b_i[$,
  \item $\lim_{y\rightarrow \pm\infty} h_i(y)=+\infty$;
  \item $h_{2i}(t)=h_{2i+1}(t-a_{2i}+a_{2i+1})$, for any $a_{2i}<t<b_{2i}$;
  \item $h_{2i}(t)=h_{2i+3}(t-a_{2i}+a_{2i+3})$, for any $t<a_{2i}$;
  \item $h_{2i-t}(t)=h_{2i}(t-b_{2i-1}+b_{2i})$, for any $t>b_{2i-1}$.
  \end{itemize}
 \end{enumerate}
\end{prop}

\begin{proof} 
The first case is a direct consequence of Proposition \ref{prop_llgeod} and the fact that, in this case, the saturation of any lightlike geodesic of the universal is the entire space. The interval $I$ corresponds to the interval of definition of the geodesic. 

Let us assume now that $K$ is not periodic. Let $\eta_1$ be a lightlike geodesic such that $g(\dot \eta_1,K)=1$. It follows from Proposition \ref{prop_llgeod} that the map $(s,t)\mapsto \Phi_K^t(\eta_1(s))$ is a diffeomorphism onto its image, that we denote $U_1$. In these coordinates the metric reads 
$h_1 dt^2+2dsdt$. Let $U_1^+$ be a connected component of $\alpha^{-1}(]0,+\infty[)\cap U_1$. 
Let $\eta_2$ be another geodesic such that $g(\dot\eta_2,K)=1$ and cutting $\eta_1$ at a point $p\in U_1^+$. We define $U_2$ and $h_2$ as above and $U_2^+$ as the connected component of  $\alpha^{-1}(]0,+\infty[)\cap U_2$ that contains $p$. It is easily verified that $\eta_2$ cuts all the leaves of $K$ contained in $U_1^+$ (cp. previous proof), therefore $U_1^+=U_2^+$ and the functions $h_1$ and $h_2$ coincide on $U_1^+$. It implies that the derivative of  $h_1$ on the boundary of $U_1^+$ in $U_1$ is equal to the derivative of  $h_2$ on the boundary of $U_2^+$ in $U_2$. According to Proposition \ref{prop_llgeod} it means that $h_1$ changes sign if and only if $h_2$ does. As any pair of points can be connected by a broken lightlike geodesic, it implies that if a function $h_i$ takes negative values they all do. 

The properties of the function $h_i$ are also given by Proposition \ref{prop_llgeod}. The fact that the atlas is finite is equivalent to the fact that $K$ has only a finite number of lightlike orbits and therefore follows from Proposition \ref{prop_cases}.  The identities between the $h_i$'s follow from the 
fact that the transition maps between the charts are isometries.
\end{proof}

Let us remark that Proposition \ref{prop_atlases} actually says that there are only three possible dynamics for Killing fields 
of spacelike Zoll cylinders, the three dynamics that appear on de Sitter space. The study thus splits in three cases that we will call elliptic, parabolic and hyperbolic in reference to the constant curvature case.
 In order to be able to determine when such metrics are indeed spacelike Zoll, we have made assumptions of the $h_i$  appearing Proposition \ref{prop_atlases}, see sections \ref{section_parabolic} and \ref{section_hyperbolic}.

\section{The Conformal Classes}\label{sec_conf}

In this section we prove the $C^0$-classification of the conformal classes of spacelike Zoll cylinders admitting a Killing vector field. 

\begin{defi}
Let $g, g'$ be Lorentzian metrics on a manifold $M$ and $\Phi\colon M\to M$ a homeomorphism. The application $\Phi$ is called a conformal homeomorphism 
if it maps $g$-lightlike geodesics to $g'$-lightlike geodesics up to parameterization. If such a $\Phi$ exists $(M,g)$ and $(M,g')$ are called $C^0$-conformal. 
\end{defi}

Note that for surfaces $C^0$-conformality is equivalent to the property that the lightlike foliations are mapped onto each other. Denote with $[g]$ the conformal class of the pseudo-Riemannian metric $g$.

\begin{thm}\label{thm_conf_killing}
Let $(C,g)$ be a spacelike Zoll cylinder with a non trivial Killing vector field $K$. Then $(C,g)$ is $C^0$-conformal to the $k$-fold cover of de Sitter space, 
where $2k$ is the number of intersection points  between any pair of distinct spacelike geodesics. Besides, if $K$ is periodic then $(C,g)$ is $C^\infty$-conformal 
to the $k$-fold cover of de Sitter space.
\end{thm}

The proof will be given at the end of the section.
In general the $C^0$-conformality cannot be improved to $C^2$-conformality  as Theorem \ref{theo_nonsmoothconformal} shows.

\begin{prop}\label{prop_extend}
Let $(C,g)$ be a globally hyperbolic spacetime admitting a conformal embedding $F\colon(C,g)\to (S^1\times \R,d\varphi^2-dt^2)$. Assume that there exists a 
conformal homeomorphism $\Phi\colon (C,[g])\to (C,[g])$ that leaves each lightlike foliation of $(C,g)$ invariant. Then $F\circ \Phi \circ F^{-1}$ has a unique 
extension as a conformal homeomorphism of $(S^1\times \R,d\varphi^2-dt^2)$. If furthermore $\Phi$ is a $C^n$-diffeomorphism, so will be the extension.
\end{prop}

\begin{proof}
Since we are interested in the conformal structure only, we can assume from the very beginning that $C$ is an open subset of $S^1\times \R$ bounded by 
possible none, one or two simply closed non timelike loops. 

Consider $(S^1\times \R,d\varphi^2-dt^2)$ as the quotient of $(\R^2,dxdy)$ by the $\Z$-action generated by $(x,y)\mapsto (x+\sqrt{2}\pi,y+\sqrt{2}\pi)$. Lift 
everything to $\R^2$ and denote the lift of $\Phi$ with $\widetilde\Phi$. Since $\widetilde\Phi$ maps horizontal lines to horizontal lines and vertical lines to
vertical lines, we see that $\widetilde\Phi(x,y)=(\widetilde\Phi_1(x),\widetilde\Phi_2(y))$. 
By the assumption that $(C,g)$ is globally hyperbolic the intersection of any lightlike line in $(\R^2,dxdy)$ with $\widetilde{C}$ is an interval. This implies that 
the maps $\widetilde\Phi_1(x)=x\circ \widetilde\Phi$ and $\widetilde\Phi_2(y)=y\circ\widetilde\Phi$ are well defined. Since $x(\widetilde C), y(\widetilde C)=\R$, 
we can define the extension of $\widetilde\Phi$ denoted by $\widetilde\Phi_e$ to $\R^2$ by setting $\widetilde \Phi_e (x,y):=(\widetilde\Phi_1(x),
\widetilde\Phi_2(y))$. This extension is unique if we impose the condition of conformality on the extension. Since $\widetilde\Phi_1$ and $\widetilde\Phi_2$ are 
equivariant under the deck transformation group of $\R^2$ over $S^1\times\R$ described above, $\widetilde\Phi_e$ descends to a conformal homeomorphism of 
$(S^1\times\R, d\varphi^2-dt^2)$. 
\end{proof}

\begin{cor}\label{C1}
If $K$ is a smooth conformal vector field on a globally hyperbolic cylinder $(C,g)$, then for every smooth conformal embedding $F\colon (C,g)\to (S^1\times \R,d\varphi^2-dt^2)$ 
there is a unique smooth extension $\overline K$ of $F_* K$ to a smooth conformal vector field of $(S^1\times \R,d\varphi^2-dt^2)$.
\end{cor}

\begin{proof}
We have seen in the previous proof that the local flow of the lift of $F_*K$ to the universal cover $(\R^2,dxdy)$ has the form $\Phi_t(x,y)=(\Phi_{1,t}(x),\Phi_{2,t}(y))$. 
This implies that the lift of $F_*K$ has the form $(K_1(x),K_2(y))$. Since the intersection of every horizontal and vertical line with the lift of $F(C)$ is non empty and 
connected, we can extend the functions $K_1$ and $K_2$ to $\R^2$ by setting $K_{(x,y)}=(K_1,K_2)$ where $K_1$ is the value of the $x$-part of $K$ on the 
intersection of the vertical line through $(x,y)$ with the lift of $F(C)$ and $K_2$ is the respective value on the intersection of the horizontal line with the lift of $F(C)$.
Since the lift of $F_*K$ is invariant under the group of deck transformations, it is now obvious that the constructed vector field induces a smooth conformal vector field 
on $(S^1\times \R,d\varphi^2-dt^2)$.
\end{proof}

\begin{cor}\label{cor_bdy_killing}
If $(C,g)$ is spacelike Zoll and admits a nontrivial Killing vector field $K$, then the conformal boundary is piecewise smooth and spacelike. If $K$ has no lightlike 
leaves then the boundary is spacelike and smooth.
\end{cor}


\begin{proof}
By Corollary \ref{C1} the Killing field $K$ admits a unique conformal extension to $S^1\times \R$ for every conformal embedding. Since the image of $C$ is 
invariant under the flow of the extension, so is the conformal boundary. Therefore the conformal boundary consists of non timelike orbits of the extension since 
$(C,g)$ is globally hyperbolic. By Proposition \ref{prop_cases} $K$ has only finitely many lightlike orbits. Therefore the conformal boundary contains only finitely 
many singularities of the extension, i.e.\ the common limit of lightlike orbits.  The rest consists of spacelike or lightlike orbits. 

We want to exclude the lightlike case. So assume that there is a lightlike orbit of $K$ in the boundary of $C$. We denote it by $I$. Since the boundary is 
invariant under the flow of $K$ the entire lightlike orbit of $K$ is contained in the boundary. 
Let $\eta$ be a lightlike geodesic asymptotic to a point in $I$. By Proposition \ref{prop_cases} we know that $g(K,K)\to\infty$ as $\eta$ approaches the boundary.
Especially $K$ will be spacelike near the boundary. We will now consider $\eta$ only near $I$. Note that since $C$ is $2$-dimensional $-g$ is Lorentzian again. 
Further $(C,-g)$ is time orientable by Lemma \ref{lem_orient}. Time orient $(C,-g)$ such that $K$ is future pointing on $\eta$. Lift everything to the universal cover
$(\widetilde{C},-\widetilde{g})$. Now denote with $J^+(y)$ and $J^-(y)$ the causal future and past respectively of $y\in \widetilde{C}$ relative to $-\widetilde g$ 
with the lifted time orientation. Since the lifted boundary is lightlike as well we see, e.g. by considering the situation in a conformal embedding into 
$(\R^2, -dxdy)$, that for points $x$ on $\widetilde \eta$ sufficiently close to the boundary the set $J^+(x)\cap  J^-(\Phi^{1}_{\widetilde{K}}(x))$ is compact  in 
$\widetilde C$, where $\Phi_{\widetilde{K}}$ denotes the flow of lifted Killing field $\widetilde K$. Recall that every Lorentzian metric on $\widetilde{C}$ is causal. 
It is well known that these two properties imply that the set of future pointing causal curves, modulo reparameterizations, from $x$ to $\Phi^{1}_{\widetilde{K}}(x)$ 
is compact in the space of causal paths of $(\widetilde{C},-\widetilde{g})$ (Proposition 8.7 in \cite{Beem-Ehrlich}). Therefore $x$ and $\Phi^1_K(x)$ are 
connected by a maximal $-g$-timelike (i.e.\ $g$-spacelike) geodesic of $g$-length at least $\int_{0}^1\sqrt{g(K,K)}= \sqrt{g(K,K)}(x)$. The right hand side 
diverges as $x\to\partial C$, thus showing that $(C,g)$ contains arbitrarily long non selfintersecting spacelike geodesic arcs. This contradicts Wadsley's Theorem
(cp. the last argument in the proof of Proposition \ref{prop_cases}).
\end{proof}

%
%
%
%

\begin{defi}
\begin{itemize}
\item[(a)] Let $(C,g)$ be  globally hyperbolic cylinder and $F\colon (C,g)\to (S^1\times \R,d\varphi^2-dt^2)$ a conformal embedding. 
A ping-pong in $(\overline{F(C)},d\varphi^2-dt^2)$ is a piecewise smooth  simply closed lightlike loop with vertices on the boundary.
\item[(b)] Let $k\in\N$. A globally hyperbolic cylinder $(C,g)$ has the $k$-ping-pong-property ($k$-PPP) if every lightlike geodesic 
of $(C,g)$ lies on a ping-pong and  every ping-pong has exactly $2k$ vertices.
\end{itemize}
\end{defi}

\begin{rema}
Ping-pongs can only exist in conformally compact globally hyperbolic cylinders.  Further, ping-pongs are invariant under conformal homeomorphisms, i.e.\ the definition is independent of the conformal embedding $F$. 


It is clear from the construction of the conformal class of de Sitter that the $k$-fold cover of de Sitter has the $k$-PPP.
\end{rema}

The next Proposition is the first step in the proof of Theorem \ref{thm_conf_killing}.

\begin{prop}\label{P2.3}
Let $(C,g)$ be a spacelike Zoll cylinder admitting a non trivial Killing vector field. Then $(C,g)$ has the $k$-PPP where $2k$ is the number of intersections
of any pair of spacelike geodesics. 
\end{prop}

Note that finite covers of de Sitter show that every $k\in\N$ appears. 

\begin{lem}\label{lem_pp_unique}
If the conformal boundary of a globally hyperbolic cylinder in $S^1\times \R$ has no lightlike parts, then every lightlike geodesic lies on at most one ping-pong.
 Further if every lightlike geodesic lies on a ping-pong, then the spacetime has the $k$-PPP for some $k\in\N$.
\end{lem}

\begin{proof}
If the conformal boundary has no lightlike parts the intersection of a lightlike line with it is unique. Therefore the vertices and sides of a ping-pong are uniquely 
determined by any side of it.  Further if the conformal boundary has no lightlike parts, the intersection of a lightlike geodesic with the boundary 
varies continuously with the geodesic. Therefore the first selfintersection of a ping-pong varies continuously. Since the number of sides and vertices of a 
ping-pong  is finite, it is constant throughout the set of lightlike geodesics.
\end{proof}

\begin{proof}[Proof of Proposition \ref{P2.3}]
We will show that every lightlike geodesic is a side of a ping-pong by considering it as the limit of a sequence of spacelike geodesics.

Let $F\colon (C,[g])\to (S^1\times \R,[d\varphi^2-dt^2])$ be a conformal embedding. We will not distinguish between $(C,g)$ and its image under $F$. 
Reparameterize all spacelike geodesics of $(C,g)$ as graphs over $S^1\times \{0\}$, i.e.\ graphs of $1$-Lipschitz functions on $S^1$.

Now let $\eta$ be a lightlike geodesic of $(C,g)$. Reparameterize $\eta$ as a partial graph over $S^1\times\{0\}$ and denote it with the same letter. Next 
consider a sequence of spacelike pregeodesics $\gamma_n$ such that $\dot \gamma_n(0)\to \dot\eta(0)$. By the Theorem of Arzela-Ascoli a subsequence of
$\gamma_n$ converges uniformly to a $[d\varphi^2-dt^2]$-non timelike curve $\gamma_\infty\colon S^1\to S^1\times \R$. By our assumptions $\eta$ is a subarc of 
the limit curve. 

Since $\gamma_\infty$ is the limit of spacelike pregeodesics and $(C,g)$ is spacelike Zoll, the limit curve has to be lightlike everywhere on the intersection with 
$C$. This follows from the fact that in $C$ $\gamma_\infty$ is a non timelike pregeodesic as it is a limit of spacelike pregeodesics. If it is not lightlike, $\gamma_\infty$ 
will be a spacelike pregeodesic and therefore nowhere lightlike, thus contradicting the initial assumption on the sequence. 

Fix a simply closed spacelike geodesic $\gamma_0$ of $(C,g)$ not contained in the sequence $\{\gamma_n\}_{n\in\N}$. Since all $\gamma_n$'s intersect 
$\gamma_0$ transversally in exactly $2k$ points, the limit curve intersects $\gamma_0$ in exactly $2k$ points as well. Note that the intersections cannot approach 
one another in the limit since on $\gamma_0$ the injectivity radius is bounded from below. Therefore $\gamma_\infty$ contains exactly $2k$ lightlike pregeodesics 
of $(C,g)$. 

We claim that the limit curve has only vertices on the boundary of $C$ in $S^1\times\R$. Then $\gamma_\infty$ will be a ping-pong with exactly $2k$ sides. If the 
$2k$ lightlike pregeodesics do not cover the entire limit curve, a piece of the boundary has to be part of the limit curve. Note that by Corollary \ref{cor_bdy_killing} the 
conformal boundary of $(C,g)$ consists of spacelike and constant orbits of the unique conformal extension of the $g$-Killing vector field $K$ to 
$(S^1\times\R,[d\varphi^2-dt^2])$. Let $\gamma_\infty|_{[t_0,t_1]}$ be a subarc lying in a spacelike orbit of $K$ and $U$ a neighborhood with $K|_U$ spacelike. By restricting 
$U$ and $[t_0,t_1]$ we can assume that $g(K,K)|_{\gamma_n}$ has at most one critical point, a maximum, in $U$ for all $\gamma_n$ intersecting $U$. This follows from the 
classification of the Killing vector fields of spacelike Zoll cylinders in Proposition \ref{prop_atlases}. In fact let $t$ be a critical point of $g(K,K)|_{\gamma_n}$ that is not a 
maximum. Then $\gamma_n$ is transversal to $K$. Thus the $K$-orbit through $\gamma_n(t)$ is itself a geodesic. If it is spacelike, it has to be closed and $K$ is spacelike 
everywhere. In this case there is only one geodesic $K$-orbit and we can assume that it lies outside of $U$. In the other cases $K_{\gamma_n(t)}$ has to be a non spacelike 
and again we can assume that it is disjoint from $U$. Further note that the maxima of $g(K,K)|_{\gamma_n}$ are exactly the minima of 
$$\arcosh \angle_{hyp} (K,\dot{\gamma}_n)=\frac{g(K,\dot{\gamma}_n)}{\sqrt{g(K,K)}\sqrt{g(\dot{\gamma}_n,\dot{\gamma}_n)}}.$$
Note that this definition makes sense without referring  to $g$, since it coincides with the respective quotient in $(S^1\times \R,d\varphi^2-dt^2)$.
Consequently, by Proposition \ref{prop_atlases}, the quotient is monotonous in $U$ except at its minima. 
From our assumptions we know that $\gamma_n|_{[t_0,t_1]}$ converges uniformly to $\gamma_\infty|_{[t_0,t_1]}$ a piece of a spacelike orbit of $K$. Use $K$ and a
curve orthogonal to $K$ in $U$ to introduce coordinates $(w,z)$ on $U$ such that $\partial_w=K$ and $\partial_z\perp K$ relative to $d\varphi^2-dt^2$. Note that on the 
intersection with $C$ the orthogonality also holds with respect to $g$. Choose constant $0<C_1,C_2<\infty$ such that the absolute value of the slope of lightlike vectors in these coordinates is
bounded between $C_1$ and $C_2$. Reparameterize the $\gamma_n$ and $\gamma_\infty$ on the intersection with $U$ as graphs over the $w$-axis. Let $w_0:=\gamma_\infty (t_0)<
w_1:=\gamma_\infty(t_1)$. For $\varepsilon>0$ choose $N$ such that 
$$\gamma_n|_{[w_0,w_1]}\subset U\cap \{|z|<\varepsilon C_2(w_1-w_0)\}$$ 
for all $n\ge N$. Since the slope of $\dot\gamma_n$ is bounded by $C_2$ and $z(\gamma_n)$ has at most one critical point in $U$, the slope of $\dot\gamma_n$
is bounded by $C_2 \varepsilon$ on a set $A\subseteq [w_0,w_1]$ of measure at least $(w_1-w_0)(1-2C_2\varepsilon)$. 

Now we can give a bound on $\frac{g(K,\dot{\gamma}_n)}{\sqrt{g(K,K)}\sqrt{g(\dot{\gamma}_n,\dot{\gamma}_n)}}$ on $A$. In fact we know that $d\varphi^2-dt^2$ in the 
$(w,z)$-coordinates reads as $Edw^2-Gdz^2$ for some positive smooth functions $E,G$ on $U$. The upper bound on the slope of the lightlike directions is equivalent
to saying $E-G C_2^2\le 0$, i.e.\ $\frac{G}{E}\ge \frac{1}{C_2^2}$. The lower bound on the slope is equivalent to saying $E-GC_1^2\ge 0$, i.e.\ $\frac{G}{E}\le \frac{1}{C_1^2}$.
For $\dot\gamma_n=(1,\dot\gamma_{z,n})$ we then have 
$$1-\frac{G}{E}\dot\gamma_{z,n}^2\ge 1-\frac{\dot\gamma_{z,n}^2}{C_1^2}\ge 1-\frac{\e^2 C_2^2}{C_1^2}$$
on $A$. Consequently 
$$\frac{g(K,\dot{\gamma}_n)}{\sqrt{g(K,K)}\sqrt{g(\dot{\gamma}_n,\dot{\gamma}_n)}}=\frac{1}{\sqrt{1-\frac{G}{E}\dot\gamma_{z,n}^2}}\le \frac{C_1}{\sqrt{C_1^2-C_2^2\varepsilon^2}}$$
on $A$.

By the choice of coordinates the $g$-gradient of $w$ on the intersection with $C$ is grad$^g w=\frac{K}{g(K,K)}$. Therefore we know that 
$$\frac{\sqrt{g(K,K)}}{\sqrt{g(\dot\gamma_n,\dot\gamma_n)}}dw(\dot\gamma_n)=\frac{g(K,\dot\gamma_n)}{\sqrt{g(K,K)}\sqrt{g(\dot\gamma_n,\dot\gamma_n)}}
\le \frac{C_1}{\sqrt{C_1^2-C_2^2\varepsilon^2}}$$
or equivalently
$$\frac{\sqrt{C_1^2-C_2^2\varepsilon^2}}{C_1}\sqrt{g(K,K)}\le \sqrt{g(\dot\gamma_n,\dot\gamma_n)}.$$
Thus we have 
$$L^g(\gamma_n)\ge \int_A \sqrt{g(\dot\gamma_n,\dot\gamma_n)} dt \ge (w_1-w_0)(1-2C_2\varepsilon)\frac{\sqrt{C_1^2-C_2^2\varepsilon^2}}{C_1}\inf_U\sqrt{g(K,K)}.$$
Since we can choose $U$ as small as we wish, $\inf_U\sqrt{g(K,K)}$ will diverge to $\infty$ by Proposition \ref{prop_atlases}. Thus the $g$-length 
of the $\gamma_n$ diverges as $n\to\infty$. This contradicts the corollary of 
Waldsley's theorem asserting that the geodesic flow on the unit tangent bundle of a spacelike Zoll manifold is periodic.
\end{proof}

\begin{prop}\label{prop_conf}
A globally hyperbolic cylinder $(C,g)$ has the $k$-PPP  and the conformal boundary contains no lightlike parts iff it is $C^0$-conformal to the $k$-fold
cover of de Sitter space. Further if the conformal boundary is $C^n$-spacelike, then the conformal homeomorphism can be chosen to be a $C^n$-diffeomorphism.
\end{prop}

Assume that the globally hyperbolic cylinder $(C,g)$ has the $k$-PPP and the conformal boundary contains no lightlike parts. Lift everything to the universal
cover $(\R^2,dxdy)$ of $(S^1\times \R,d\varphi^2-dt^2)$ with the deck transformation group generated by $(x,y)\mapsto (x+\sqrt{2}\pi,y+\sqrt{2}\pi)$. 
Then the boundary of the universal cover $\widetilde{C}$ is the union of the graphs of two $\sqrt{2}\pi$-equivariant homeomorphisms 
$\theta^\pm\colon \R\to \R$, i.e.\ $\theta^\pm (x+\sqrt{2}\pi)=\theta^\pm (x)+\sqrt{2}\pi$, over the $x$-axis ($\theta^-<\theta^+$).

\begin{lem}
Assume that the conformal boundary of the globally hyperbolic cylinder $(C,g)$ does not contain any lightlike parts. Then $(C,g)$ has the $k$-PPP iff
$((\theta^-)^{-1}\circ \theta^+)^k(x)=x+\sqrt{2}\pi$ for all $x\in \R$.
\end{lem}

\begin{proof}
Let $x\in\R$. Then $(x,\theta^+(x))$ is the future endpoint of a vertical lightlike $\widetilde g$-geodesic $\gamma^+_{x}$ of $(\widetilde{C},\widetilde{g})$ 
in $\R^2$. The point 
$(\theta^-)^{-1}\circ \theta^+(x),\theta^+(x))$ is the past endpoint of the horizontal lightlike $\widetilde g$-geodesic $\gamma^-_{\theta^+(x)}$ of 
$(\widetilde C,\widetilde g)$ in $\R^2$ whose future endpoint in $\R^2$ is $(x,\theta^+(x))$. Now we can consider the vertical lightlike geodesic of 
$(\widetilde C,\widetilde g)$ whose past endpoint is $(\theta^-)^{-1}\circ \theta^+(x),\theta^+(x))$ and start the above construction over again. 
This defines inductively a series of wedges in $(\R^2,dxdy)$ with vertices in $\partial \widetilde C$ and sides in $\widetilde C$. 

Now if $(C,g)$ has the $k$-PPP take the lift of a ping-pong that contains a given lightlike $\widetilde g$-geodesic $\gamma^+_x$. The ping-pong in 
$(C,g)$ returns to the same geodesic after $k$ wedges in $(S^1\times \R,d\varphi^2-dt^2)$. By the first paragraph this implies that 
$((\theta^-)^{-1}\circ \theta^+)^k(x)=x+\sqrt{2}\pi$. Since any $\gamma_x$ lies on the lift of a ping-pong, we see that the $k$-PPP implies the 
identity for $((\theta^-)^{-1}\circ \theta^+)^k$. 

For the other direction we can restrict ourself to geodesic lifting to vertical lightlike geodesics since the claim for geodesics lifting to 
horizontal lightlike geodesics follows by considering the vertical lightlike geodesic with the same future endpoint as the given horizontal lightlike geodesic.
If the identity $((\theta^-)^{-1}\circ \theta^+)^k(x)=x+\sqrt{2}\pi$ holds for all $x$ then the projection to $S^1\times \R$ of the wedges 
constructed in the first paragraph will form a $k$-ping-pong, thus showing the lemma.
\end{proof}

\begin{proof}[Proof of Proposition \ref{prop_conf}]
The second assertion will readily follow from the construction in the first part. Further if $(C,g)$ is $C^0$-conformal to the $k$-fold cover of de Sitter space, then
the $k$-PPP is obvious for $(C,g)$. The conformal boundary does not contain any lightlike parts either since this is invariant under conformal homeomorphisms.

Using the Lemma choose a $\sqrt{2}\pi$-equivariant homeomorphism $\psi\colon\R\to\R$ conjugating $(\theta^-)^{-1}\circ \theta^+$ to a translation 
by $\frac{\sqrt{2}\pi}{k}$. Applying $(\psi^{-1}\circ\theta^+\circ\psi)^{-1}$ to both sides we obtain 
$$(\psi^{-1}\circ\widetilde{\theta}^-\circ\psi)^{-1}(x)-(\psi^{-1}\circ\widetilde{\theta}^+\circ\psi)^{-1}(x)=\frac{\sqrt{2}\pi}{k}.$$
Now we can isotope $\psi^{-1}\circ\widetilde{\theta}^-\circ\psi$ and $\psi^{-1}\circ\widetilde{\theta}^+\circ\psi$ simultaneously to 
translations. Note that for the $k$-fold cover of de Sitter the boundary is given by two translations whose difference is $\frac{\sqrt{2}\pi}{k}$. Thus the result of this 
isotopy is a conformal homeomorphism of $(C,g)$ to the $k$-fold cover of de Sitter space.

 Finally the conformal boundary is $C^n$-spacelike if, and only if the homeomorphisms $\theta^\pm$ are $C^n$-diffeomorphisms. Since 
$(\theta^-)^{-1}\circ \theta^+$ is periodic the conjugation $\psi$ can be chosen to be $C^n$ as well. This shows that in this case $(C,g)$ is $C^n$-conformal
to the $k$-fold cover of de Sitter space.
\end{proof}

%
%
%

\begin{proof}[Proof of Theorem \ref{thm_conf_killing}]
The proof is follows from Corollary \ref{cor_bdy_killing}, Proposition \ref{P2.3} and Proposition \ref{prop_conf}. In fact if $(C,g)$ is spacelike Zoll then 
by Corollary \ref{cor_bdy_killing} the conformal boundary does not contain any lightlike parts. Further by Proposition \ref{P2.3} $(C,g)$ has the $k$-PPP
for some $k\in \N$. By Proposition \ref{prop_conf} these two properties imply that $(C,g)$ is $C^0$-conformal to the $k$-fold cover of de Sitter space. 
If $K$ is periodic then by Corollary \ref{cor_bdy_killing} the boundary is smooth and by Proposition \ref{prop_conf} the conformal homeomorphism is a smooth 
diffeomorphism.
\end{proof}


\section{Parabolic case}\label{section_parabolic}
In this section we will describe a family of parabolic spacelike Zoll surfaces, i.e.\ admitting an atlas similar to the one described at point 2 of Proposition 
\ref{prop_atlases}. This family will allow us to construct several interesting examples, such as spacelike Zoll M\"obius strip with non 
constant curvature or spacelike Zoll cylinder not smoothly conformal to a cover of de Sitter space. We start with following definition  of a ``parabolic atlas".
\begin{defi}\label{def_para}
Let $(C,g)$ be a Lorentzian cylinder  with an atlas  $\mathcal A=\{(U_i,\phi_i); i\in\zkk\}$. We denote by $\Phi_{ij}=\phi_{j}\circ\phi^{-1}_{i}$ the transition functions of $\mathcal A$.\\
We will say that $\mathcal {A}$ is a parabolic atlas of  $(C,g)$ if:
\begin{enumerate}
 \item for all $i\in\zkk$, the image of $\phi_i$ is $\R^2$;
 \item the transition functions are the following:
 $$
 \begin{array}{lllll}
 & \Phi_{2i,2i+1}:& H^+&\rightarrow &H^+\\
 & & (x,y)&\mapsto& \left(-x+\frac {2}{y},y\right),\\
 \text{if}\ i\neq 0& \Phi_{2i-1,2i}:& H^-&\rightarrow& H^-\\
 & &(x,y)&\mapsto&  \left(-x+\frac {2}{y},y\right),\\
 & \Phi_{2k-1,0}:& H^-&\rightarrow& H^-\\
 & &(x,y)&\mapsto & \left(-x+\frac {2}{y}+\tau,y\right),
  \end{array}
 $$
 where $H^+=\{(x,y)\in \R^2;y>0\}$, $H^-=\{(x,y)\in \R^2; y<0\}$ and $\tau\in \R$;
 \item for all $i\in\{1,\dots 2k\}$, $$g_i=\phi_i^{-1}{}^*g= y^2 dx^2+2dxdy+f_i(y)dy^2,$$
 where $f_i$ is a smooth function satisfying $1-y^2f_i(y)>0$ for all $y\in \R$.
\end{enumerate}
\end{defi}
\begin{figure}[ht]
 \centering
\includegraphics[scale=.4]{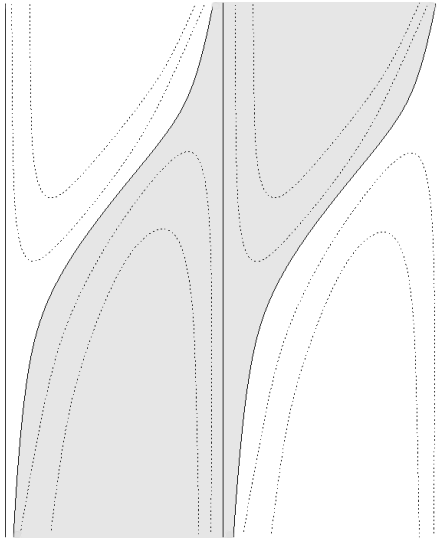}
\caption{The manifold, the open set $U_1$ and the Killing field ($k=2$)}
\end{figure}

\begin{rema}
Note that a parabolic atlas induces an analytic structure on $C$. The Killing field $K$ is according to the conditions analytic as well. In opposition the 
metric $g$ need not be analytic, but the $g$-length of $K$ is again an analytic function on $C$. 
\end{rema}

Clearly,  Lorentzian cylinders admitting a parabolic atlas posses a Killing vector field $K$ that is everywhere spacelike except  on a finite number of lightlike orbit 
($K$ reads as $\partial_x$ in any map $\phi_i$). DeSitter space clearly admits such an atlas. It has  the following 
parameters $k=1$, $\tau=0$ and  $f_1=f_2=0$. Let us note that if we modify only $\tau$, the cylinder obtained still has constant curvature equal to $1$ but  is 
not isometric to  de Sitter space (for example its spacelike geodesics are no more closed).

We remark also that for any $i\in \zkk$ the restrictions of $f_{2i}$ and $f_{2i+1}$ to $H^+$ have to be equal as well as the restrictions of  $f_{2i-1}$ and $f_{2i}$ to $H^-$. In particular if $g$ is analytic then $f_1=f_2=\dots=f_{2k}$.

\begin{prop}\label{prop_parabolicatlas}
Let $(C,g)$ be a  spacelike Zoll cylinder admitting a parabolic Killing field $K$, i.e.\  that is not periodic and does not vanish. 
Let $\sigma$ be the curvature of $g$ and $\alpha\colon C\to\R$ the function defined by $\alpha(p)=g_p(K_p,K_p)$. 
There exists $l>0$ such that $(C, l\cdot g)$ admits a parabolic atlas if and only if 
for any $p\in C$, $\alpha(p)=0$ implies $\sigma(p)>0$ and $d\sigma(p)=0$. 
\end{prop}

\begin{proof} 
%
According to Proposition \ref{prop_cases}, $K$ has spacelike and lightlike orbits but no timelike ones.
Let $\eta$ be a lightlike geodesic of $(C,g)$  transversal to $K$. According to  Proposition \ref{prop_llgeod}, the function  $\alpha$ vanishes somewhere on $\eta$ and goes to $+\infty$ on both ends of $\eta$.

Let $U$ be the saturation of $\eta$ by $K$. There exists a lightlike geodesic vector field $Y$ on $U$ tangent to $\eta$ such that $[Y,K]=0$. It allows us to find 
coordinates on $U$ such that $g$ reads as ${h}(y)dx^2+2 dxdy$ with ${h}\geq 0$ for $(x,y)\in \R\times I$, according to Proposition \ref{prop_atlases} ${h}$ vanishes at
only one point. The assumption on the curvature implies that ${h}(y)=0$ implies ${h}''(y)>0$. Choose the coordinates so that ${h}(0)=0$ and denote by $a$ the function 
defined by ${h}(y)=y^2e^{2a(y)}.$ Rescaling $g$ we assume that $a(0)=0$. The hypothesis $\sigma'(0)=0$ entails that $a'(0)=0$. 

Let $\gamma(t)=(x(t),y(t))$ be the unique curve satisfying:
\begin{equation}\label{system_gamma}\left\{\begin{array}{r}
 g(\gamma'(t),\partial_x)=1\\
 \gamma'(t).ye^{a(y)}=1\\
 \gamma(0)=0 
 \end{array} \right.
 \end{equation}
 i.e.
 $$\left\{\begin{array}{r}
 x'y^2e^{2a(y)}+y'=1\\
 y'(t)e^{a(y)}(1+a'(y)y)=1\\
 \gamma(0)=0 
 \end{array}\right.$$
\begin{fact}\label{lefact}
 For all $y\in \R$, we have $1+a'(y)y\neq 0$.
\end{fact}
Indeed, if we have $1+a'(y)y= 0$ then the curve $t\mapsto (t,y)$ is a complete spacelike geodesic. As it is not closed, it contradicts the fact that the metric is spacelike Zoll. 

Thanks to Fact \ref{lefact}, we can write
\begin{eqnarray*}
   y'&=&\frac{e^{-a(y)}}{1+ya'(y)}\\
   x'&=&\frac{e^{-2a(y)}}{1+ya'(y)}\frac{1+y a'(y)-e^{-a(y)}}{y^2}
  \end{eqnarray*}
Therefore 
$$\frac{\partial x}{\partial y}=e^{-a(y)}\frac{1}{y^2}\left(1+ya'(y)-e^{-a(y)}\right)$$
as $a'(0)=0$ we see that  $\frac{\partial x}{\partial y}$ is well defined on $\R$ and smooth. It implies that $\gamma$ is the graph of a function, in particular it cuts each horizontal line exactly once. The fact that ${h}$ goes to infinity on both ends of $\eta$ says that $\gamma$ is defined  on $\R$.

Hence, the map $\Phi:(u,v)\mapsto \gamma(v)+(u,0)$ is a smooth diffeomorphism. Equation \eqref{system_gamma} exactly says that $\Phi^*g|_U$ has the desired form. Repeating this construction for any lightlike geodesic transverse to $K$ gives us  an atlas of $(C,g)$ such that the metric has the right expression. The last things to check are the transition functions.

If $\eta'$ is another lightlike geodesic and if we denote by $V$ its saturation, then there are  3 possibilities: either $U=U'$, either $U\cap U'=\emptyset$ or $V:=U\cap U'$ is an half plane of $U$ ($\{v>0\}$ or $\{v<0\}$). The only case to deal with is the third. In that case there exists a geodesic $\delta$ in $V$ that is  perpendicular to $K$  and such that the orthogonal symmetry relatively to $\delta$  sends $\eta$ on $\eta'$ (see \cite{BavardMounoud} for details). It is not difficult to check that this symmetry is the transition function we were looking for. It has the right expression up to a possible horizontal translation. However 
it is not difficult to modify the atlas so that these translations are trivial except one.

Reciprocally, it is easily checked that a Lorentzian cylinder admitting a parabolic atlas satisfies  the conditions on the curvature given in the statement.
\end{proof}
\begin{rema}
If $(C,g)$ is spacelike Zoll cylinder admitting a parabolic atlas $\mathcal A$, then the parameter $k$ of $\mathcal A$ is equal to the number $k$ in Theorem \ref{thm_conf_killing}.

 At the moment we do not know now if there exists a spacelike Zoll metric with a parabolic Killing field that  does not 
 admit a parabolic atlas.
\end{rema}
We are able to describe all the spacelike Zoll surfaces admitting a parabolic atlas:
\begin{thm}\label{theo_parabolic}
Let $(C,g)$ be a Lorentzian cylinder admitting a  parabolic atlas  $\mathcal A=\{(U_i,\phi_i); i\in \zkk \}$. 
The surface $(C,g)$ is spacelike Zoll if and only if the parameter $\tau$ of $\mathcal A$ vanishes and there exist  $k$ smooth functions  
$\kappa_0,\dots \kappa_{k-1}\colon \R\to\R$ 
such that 
 \begin{enumerate}
  \item\label{one} for all $t\in \R$, for all $j\in \zk$, $\kappa_j(t)\geq -1$;
  \item\label{two} all the functions $\kappa_j$ have  the same infinite Taylor expansion at $0$ and satisfy $\kappa_j(0)=\kappa_j'(0)=0$;
  \item\label{three} the function $\sum_{j} \kappa_j$ is odd;
  \item for all $i\in \zkk$ the function $f_i$ such that 
    $$g_i=\phi_i^{-1}{}^*g= y^2 dx^2+2dxdy+ f_i(y)dy^2,$$ 
    satisfies 
    $$f_i(y)=\left\{\begin{aligned}
                     \frac{1-(1+\kappa_{\lfloor i/2 \rfloor})^2}{y^2}&\ \text{if}\ y>0\\
                     \frac{1-(1+\kappa_{\lceil i/2 \rceil})^2}{y^2}&\ \text{if}\ y<0,
                    \end{aligned}
   \right.$$
   where $\lfloor.\rfloor$ (resp. $\lceil.\rceil$) is the lower (resp. upper) integral part.
 \end{enumerate}
\end{thm}
%
\begin{proof}[Proof of Theorem \ref{theo_parabolic}.]
Let $\mathcal A$ be a parabolic atlas of $(C,g)$. We denote as above by $g_i$ the expression of $g$ in the coordinates $(U_i,\phi_i)$.  We recall that there exist functions $f_i$ such that the $g_i$'s read as $y^2 dx^2+2dxdy+ f_i(y)dy^2$.

\begin{lem}\label{lem_tangentparabolic}
 Let $\gamma_i\,:\,t\mapsto(x(t),y(t))$ be a unit spacelike geodesic of $(\R^2,g_i)$. There exists $c>0$ such that $\gamma_i$ is contained between the lines $y=c$ and $y=-c$ and is tangent exactly once to each of these lines. Moreover the geodesic segment between these points satisfies
 $$\frac{\partial x}{\partial y}= \frac{c\sqrt{1-y^2f_i(y)}-\sqrt{c^2-y^2}}{y^2\sqrt{c^2-y^2}}.$$
\end{lem}

\begin{proof}[Proof of Lemma \ref{lem_tangentparabolic}]
Let $\gamma_i\,:\,t\mapsto(x(t),y(t))$ be unit spacelike geodesic of $(\R^2,g_i)$. It is well know that Killing vector fields induce first integrals for the geodesic field, more precisely  the fact that the vector field $\partial_x$ is Killing implies that $g_i(\partial_x,\gamma'_i)$ is constant. Therefore, there exists  $c\geq 0$ and $\epsilon_1\in \{\pm1\}$ such that:
\begin{equation} 
 \begin{cases}
 y^2x'+y'=\epsilon_1 c\\
 y^2x'^2+2x'y'+f_i(y)y'^2=1,
\end{cases}
\end{equation}
This system of equations can be solved if and only if $c^2-y^2\geq 0$ proving that $c\neq 0$ and $-c\leq y\leq c$. Solving it we find:
\begin{eqnarray*}
  x'&=&\frac{\epsilon_1c(1-y^2f_i(y))+\epsilon \sqrt{(1-y^2f_i(y))(c^2-y^2)}}{y^2(1-y^2f_i(y))}\\
  y'&=&-\epsilon \sqrt{\frac{c^2-y^2}{1-y^2f_i(y)}},
  \end{eqnarray*}
  where $\epsilon\in \{\pm 1\}$. It implies that 
  $$\frac{\partial x}{\partial y}= \frac{-\epsilon\epsilon_1 c\sqrt{1-y^2f_i(y)}-\sqrt{c^2-y^2}}{y^2\sqrt{c^2-y^2}}.$$
The number $\epsilon_1$ determines the orientation of the geodesic and $\epsilon$ changes only when $y=\pm c$.

The fact that for any $y_0$ such that $0<|y_0|<c$ the integral
$$\int_0^{y_0} \frac{-c\sqrt{1-y^2f_i(y)}-\sqrt{c^2-y^2}}{y^2\sqrt{c^2-y^2}}dy$$
diverges implies that $\gamma$ is tangent at most once to each line $y=\pm c$.

The fact that for any $y_0\in ]0,c[$ and any $y_1\in ]-c,0[$  the integrals
\begin{eqnarray*}
 \int_{-c}^{c} \frac{c\sqrt{1-y^2f_i(y)}-\sqrt{c^2-y^2}}{y^2\sqrt{c^2-y^2}}dy\\
 \int_{y_0}^c \frac{-c\sqrt{1-y^2f_i(y)}-\sqrt{c^2-y^2}}{y^2\sqrt{c^2-y^2}}dy\\
 \int_{-c}^{y_1}\frac{-c\sqrt{1-y^2f_i(y)}-\sqrt{c^2-y^2}}{y^2\sqrt{c^2-y^2}}dy
\end{eqnarray*}
converge implies that $\gamma$ is tangent at least once to each line $y=\pm c$. \end{proof}
%
\begin{prop}\label{prop_recoll}
 Let $\gamma$ be a unit spacelike geodesic of $(C,g)$. The geodesic $\gamma$ is closed if and only if 
 $$ \int_{-c}^c c\frac{\sum_{i\in\zkk} (\sqrt{1-y^2f_i(y)}-1)}{y^2\sqrt{c^2-y^2}}dy=-\tau,$$
 where $\tau$ is the term of translation appearing in $\Phi_{2k-1,0}$ and $c=|g(\gamma',K)|$.
\end{prop}
\begin{proof}[Proof of Proposition \ref{prop_recoll}.]
Let $\gamma$ be a geodesic of $g$ and $p_0,\dots,p_l,\dots$ the points where $\gamma$ is tangent to the Killing field $K$ (such points exists according to Lemma \ref{lem_tangentparabolic}). For each geodesic segment $[p_l,p_{l+1}]$, there exist an open set $U_{i_l}$ containing it. Clearly, $i_{l+1}=i_l+1$ and $i_{0}=i_{2k}$. Without loss of generality we can suppose $i_l=l \mod 2k$.

If $p_0$ has coordinates $(x_0,-c)$ on $U_{0}$ then the coordinates of $p_1$ on $U_{0}$ are 
$$
\left(x_0+\int_{-c}^{c} \frac{c\sqrt{1-y^2f_{0}(y)}-\sqrt{c^2-y^2}}{y^2\sqrt{c^2-y^2}}dy,c\right)
$$
It follows that the coordinates of $p_1$ on $U_{1}$ are 
$$\left(-\left[x_0+\int_{-c}^{c} \frac{c\sqrt{1-y^2f_{0}(y)}-\sqrt{c^2-y^2}}{y^2\sqrt{c^2-y^2}}dy\right]+\frac{2}{c},c\right)$$
and the coordinates of $p_2$ on $U_1$ are (remark that the orientation of $\gamma$ changed)
$$\left(-\left[x_0+\int_{-c}^{c} \frac{c\sqrt{1-y^2f_{0}(y)}-\sqrt{c^2-y^2}}{y^2\sqrt{c^2-y^2}}dy\right]+\frac{2}{c}-\int_{-c}^{c} \frac{c\sqrt{1-y^2f_{1}(y)}-\sqrt{c^2-y^2}}{y^2\sqrt{c^2-y^2}}dy,-c\right)$$
We can continue the same way, in order to  obtain the coordinates of the points $p_l$ on $U_{l-1}$ and $U_{l}$. In particular, we see that the coordinates of $p_{2k}$ on $U_0$ are
$$\left(x_0+\sum_{l=0}^{2k-1} \int_{-c}^{c} \frac{c\sqrt{1-y^2f_{l}(y)}-\sqrt{c^2-y^2}}{y^2\sqrt{c^2-y^2}}dy-\frac{4k}{c}+\tau,-c\right).$$

It implies that $\gamma$ is closed if and only if $p_0=p_{2k}$ if and only if
\begin{equation}\label{eq_closed}
\sum_{l=0}^{2k-1} \int_{-c}^{c} \frac{c\sqrt{1-y^2f_{l}(y)}-\sqrt{c^2-y^2}}{y^2\sqrt{c^2-y^2}}dy=\frac{4k}{c}-\tau. 
\end{equation}

We consider now the metric $g^0$ of constant curvature $1$ given by $g^0_i=y^2dx^2+2dxdy$ and $\gamma^0$ the $g^0$-spacelike geodesic starting horizontally from $p_0$. We denote by $p_l^0$ the points of tangency of $\gamma^0$ with $K$.  Doing the same computation as above we see that the coordinates of $p_{2k}^0$ on $U_0$ are 
$$(x_0+ 2k \int_{-c}^{c} \frac{c-\sqrt{c^2-y^2}}{y^2\sqrt{c^2-y^2}}dy-\frac{4k}{c}+\tau,-c)$$ 
Computing the integral above, we find that its value is $\frac 2c$ and therefore the coordinates of $p_{2k}^0$  are in fact
$$(x_0+\tau,-c)$$
(reproving that all the spacelike  geodesics of de Sitter space are closed). In order to  finish the proof, we just replace $\frac{4k}{c}$ by $2k\int_{-c}^{c} \frac{c-\sqrt{c^2-y^2}}{y^2\sqrt{c^2-y^2}}dy$ in \eqref{eq_closed}. 
\end{proof}

\begin{lem}\label{lem_Besse}
 If $h :\R\rightarrow \R$ is a function such that the function 
 $$H:c\mapsto\int_{-c}^{c}c\frac{h(y)}{y^2\sqrt{c^2-y^2}}dy$$
 is constant, then $h$ is odd and $H=0$.
\end{lem}
\begin{proof}  We first remark that $H(c)$ only depends on the even part of $h$. Thus we will assume that $h$ is even  and prove that it has to vanish.

We define a function $I$ on $\R^+$ by 
$$I(a)=\int_0^a \frac{H(t)}{\sqrt{a^2-t^2}} dt.$$

We have 
\begin{eqnarray*}
 I(a)&=& \int_0^a \frac{2t}{\sqrt{a^2-t^2}}\int_0^t\frac{h(s)}{s^2\sqrt{t^2-s^2}}ds\, dt\\
     &=& \int_0^a\frac{2h(s)}{s^2}\int_s^a \frac{t}{\sqrt{(a^2-t^2)(t^2-s^2)}}dt\,ds\\
     &=& \int_0^a \frac{2h(s)}{s^2}ds\int_0^{+\infty}\frac{dx}{1+x^2}=\pi\int_0^a \frac{h(s)}{s^2}ds
\end{eqnarray*}
with $x=\sqrt{\frac{t^2-s^2}{a^2-t^2}}$.

Moreover if $H$ is constant, we see by direct computation that $I$ is also constant. It follows from $I'=0$ that $h=0$.\end{proof}

 For $i\in \{0,\dots,k-1\}$ we define the function $\kappa_i$ by  $\kappa_i=\sqrt{1-y^2f_{2i}(y)}-1$. These functions clearly satisfy points \ref{one} and \ref{two} of the statement. 
It follows from Lemma \ref{lem_Besse} and Proposition \ref{prop_recoll} that the geodesics of $(C,g)$ are all closed if and only if the function $c\mapsto \sum_{i\in\zkk}\kappa_i$ is odd and $\tau=0$.
\end{proof}

\begin{cor}\label{coro_parabolic}
There exist smooth M\"obius strips all of whose spacelike geodesics are closed with non constant curvature and whose orientation cover admits a 
parabolic atlas. Moreover, if the orientation cover of a non constant curvature  M\"obius strip all of whose spacelike geodesics are closed admits a parabolic atlas then it is not analytic and it is $C^0$-conformal to a $k$-cover of de Sitter with $k>1$.
\end{cor}

\begin{proof} Let $(C,g)$ be a  parabolic spacelike Zoll cylinder and $\mathcal A$ be a parabolic atlas of $(C,g)$. 
If $(C,g)$ is analytic (or if $k=1$) then the functions $\kappa_i$ given by Theorem \ref{theo_parabolic} have to be equal and therefore odd.
It follows that $(C,g)$ cannot be the lift of a metric on the M\"obius strip unless the $\kappa_i$ vanish.

Let $\kappa$ be a  smooth function on $\R$ with support on $[1,2]$ and values in $[-1,1]$. We define now three functions $\kappa_0$,$\kappa_1$ and $\kappa_2$ by $\kappa_0(t)=\kappa(t)$, $\kappa_1(t)=-\kappa(-t)$ and $\kappa_3(t)=-\kappa(t)+\kappa(-t)$.
These functions clearly verify points \ref{one}, \ref{two} and \ref{three} of the statement of Theorem \ref{theo_parabolic}.
Therefore they induce a spacelike Zoll metric $g$ on the cylinder (the parameters of the parabolic atlas being $k=3$ and $\tau=0$).

Let $\sigma:C\rightarrow C$ be the map sending $U_i$ on $U_{i+3}$ for any $i\in \Z/6\Z$ and that reads $(x,y)\mapsto (-x,-y)$ in coordinates.
This map is clearly a smooth involution with no fixed points. Moreover, despite appearances (because the orientations of the frame $(\partial_x,\partial_y)$ are opposite on $U_i$ and $U_{i+3}$), it does  not preserve the orientation. Hence $C/\sigma$ is a smooth M\"obius strip. By a direct computation, we see that the metric $g$ is invariant by $\sigma$ and therefore defines a spacelike Zoll metric  on the M\"obius strip. Further we can choose $\kappa$ such that 
the curvature of $g$ is non constant. \end{proof}
\begin{thm}\label{theo_nonsmoothconformal}
 There exists a spacelike Zoll cylinder, admitting a parabolic atlas with parameter $k>1$, that is not $C^2$-conformal to de Sitter and whose conformal boundary is not $C^2$.
\end{thm}

\begin{proof}
Let $(C,g)$ be a time-oriented spacelike Zoll admitting a parabolic atlas and let $\F_1$ and $\F_2$ be its lightlike foliations. 
We denote by $\LL_1$ and $\LL_2$ their spaces of leaves. As $(C,g)$ is globally hyperbolic, we know that $\LL_i$ are diffeomorphic to circles. The time orientation and the orientation of $(C,g)$ define an orientation on the $\LL_i$.

We define the first reflexion map $P$ (for Ping) from $\LL_1$ to $\LL_2$ that associates to a lightlike geodesic $\eta\in \LL_1$ the lightlike geodesic $\bar \eta 
\in \LL_2$ such that $\eta$ and $\bar\eta$ intersects on the future conformal boundary of $(C,g)$. It follows from the fact that  the 
boundary contains no lightlike parts, see Corollary \ref{cor_bdy_killing}, that $P$ is well defined and continuous. Actually, we do not need to have a conformal 
embedding in the flat cylinder to define the map $P$, $P(\eta)$ is the unique geodesic such that $\eta\cap P(\eta)=\emptyset$ and  any $\eta'\in \LL_1$ sufficiently closed to $\eta$ and on one side intersects $P(\eta)$.

Clearly, any smooth parametrized transversal cutting at most once each leaf of $\F_i$ defines smooth local coordinates on $\LL_i$. It follows from the definition of $P$ in terms of reflexion on the future conformal boundary that $P$ is smooth where  the boundary is spacelike and smooth. Moreover, if it is smooth but lightlike at a point then the graph of  $P$ has a horizontal or vertical tangent at this point therefore this property can be read off of $P$.
It follows from the other definition, that the regularity of $P$ is a conformal invariant of $(C,g)$.  In particular, if $(C,g)$ is conformal to a finite cover of de Sitter space then $P$ has to be smooth. Corollary \ref{cor_bdy_killing} implies that $P$ is smooth except maybe at points of $\LL_1$ corresponding to  lightlike orbits of $K$.
Let us look at $P$ at a neighborhood of such a leaf.

Let $\eta_0\in \LL_1$  that is a lightlike orbit of $K$. For example $\eta_0$ is the curve contained in $U_0$ whose equation is $y=0$. The curve $y\mapsto (0,y)$ 
is transversal to $\F_1$ and cuts each element of $\F_1$ at most once, therefore it defines local coordinates on $\LL_1$ around $\eta_0$. Possibly changing the 
time orientation, we can assume that the leaf $P(\eta_0)$ is the geodesic $\{y=0\}$ contained in $U_1$. We  define as above coordinates on $\LL_2$ around 
$P(\eta_0)$.

\begin{lem}\label{lem_PPM} Let $\kappa_0$ and $\kappa_1$ be the functions given by Theorem \ref{theo_parabolic}. For $i=0$ or $1$, let $h_i$  be the primitive vanishing at $0$ of $s\mapsto \frac{\kappa_i(s)}{s^2}$ and let $\delta_i=\lim_{s\rightarrow -\infty}h_i(s)$. Denoting by $\overline P$ the expression of $P$ in the coordinates defined above, we have for $y>0$
 $$\overline P(y)= \left\{ 
 \begin{aligned}
 F^{-1}{\left(\frac {2}{y}+\delta_1-h_0(y)\right)} & \ \text{if}\  y>0 \\
 G^{-1}{\left(-\frac{2}{y}-\delta_0+h_0(y)\right)} & \ \text{if}\ y<0
 \end{aligned}
 \right.$$
 where $F$ is the map defined for $z<0$ by  $F(z)=-\delta_1+h_1(z)-\frac{2}{z}$ and 
$G$ the map defined for $z>0$ by $G(z)=\delta_0+\frac{2}{z}-h_0(z)$. 
\end{lem}
\begin{proof}[Proof of Lemma \ref{lem_PPM}]
On $U_i$ the metric reads $y^2dx^2+2dxdy+f_i(y)dy^2$ therefore any lightlike geodesic of $U_i$ different from $\{y=0\}$ is transverse to $\partial_x$. The vector fields defined $\frac{\kappa_{\lfloor i/2 \rfloor}(y)}{y^2}\partial_x+\partial_y$ and $-\frac{2+\kappa_{\lfloor i/2 \rfloor}(y)}{y^2}\partial_x+\partial_y$ for $y>0$ and by $\frac{\kappa_{\lceil i/2 \rceil}(y)}{y^2}\partial_x+\partial_y$ and $-\frac{2+\kappa_{\lceil i/2 \rceil}(y)}{y^2}\partial_x+\partial_y$ for $y<0$  are lightlike and smooth. Therefore any lightlike geodesic of $\{y>0\}$ is the graph of a function $y\mapsto \int \frac{\kappa_{\lfloor i/2 \rfloor}(s)}{s^2}ds=h_i(y)+\opna{cst}$ or 
$y\mapsto \int -\frac{2+\kappa_{\lfloor i/2 \rfloor}(s)}{s^2}ds=\frac{2}{y}-h_{\lfloor i/2 \rfloor}(y)+\opna{cst}$ and any lightlike geodesic of $\{y<0\}$ is the graph of a function $y\mapsto \int \frac{\kappa_{\lceil i/2 \rceil}(s)}{s^2}ds=h_{\lceil i/2 \rceil}(y)+\opna{cst}$ or 
$y\mapsto \int -\frac{2+\kappa_{\lceil i/2 \rceil}(s)}{s^2}ds=\frac{2}{y}-h_{\lceil i/2 \rceil}(y)+\opna{cst}$.

Let $\eta$ be a  lightlike geodesic of $\F_1$ (the foliation that has $\{y=0\}$ as a leaf) intersecting $H_0^+$. 
It cuts $\{x=0\}$ at a point $(0,y_1)$. The image of $(0,y_1)$ by the transition function $\Phi_{0,1}$ is $(\frac{2}{y_1},y_1)$, therefore $\eta\cap H_1^+$ is the graph of a map $h_0+c_1$. As $(2/y_1,y_1)\in \eta$, we have $c_1=\frac{2}{y_1}-h_0(y_1)$. 
It follows that $\eta\cap H_1^-$ is the graph of the  map $h_1+c_1$. This implies that $\eta$ is asymptotic to the vertical line $\{x=\delta_1+c_1\}$ when $y$ goes to $-\infty$.

We now define the map $F$. Let $z$ be a negative number and $\bar \eta$ be the geodesic of $\F_2$ that contains the point $(0,z)$ of $H_1^-$.  Its intersection with $H_1^-$ is the graph of the function 
$y\mapsto \frac{2}{y}-h_1(y)+h_1(z)-\frac{2}{z}$ therefore it is asymptotic to the vertical line $\{x=-\delta_1+h_1(z)-\frac{2}{z}\}$ when $y$ goes to $-\infty$. We set therefore $F(z)=-\delta_1+h_1(z)-\frac{2}{z}$. 
By definition the  function $F$ is strictly increasing and therefore invertible.

\begin{figure}[h!]
 \labellist
 \small\hair 2pt
 \pinlabel {$\eta$} at 220 20
 \pinlabel {$\scriptstyle{(0,y)}$} at 180 215
 \pinlabel {$\scriptstyle(2/y,y)$} at  830 220
 \pinlabel {$\eta$} at 830 360 
 \pinlabel {$\eta$} at 50 360
 \pinlabel {$P(\eta)$} at 600 167
 \pinlabel {$\scriptstyle(0,\overline P(y))$} at 705 125
 \pinlabel {$\Phi_{0,1}$} at 480 270
 \endlabellist
 \centering
\includegraphics[scale=.5]{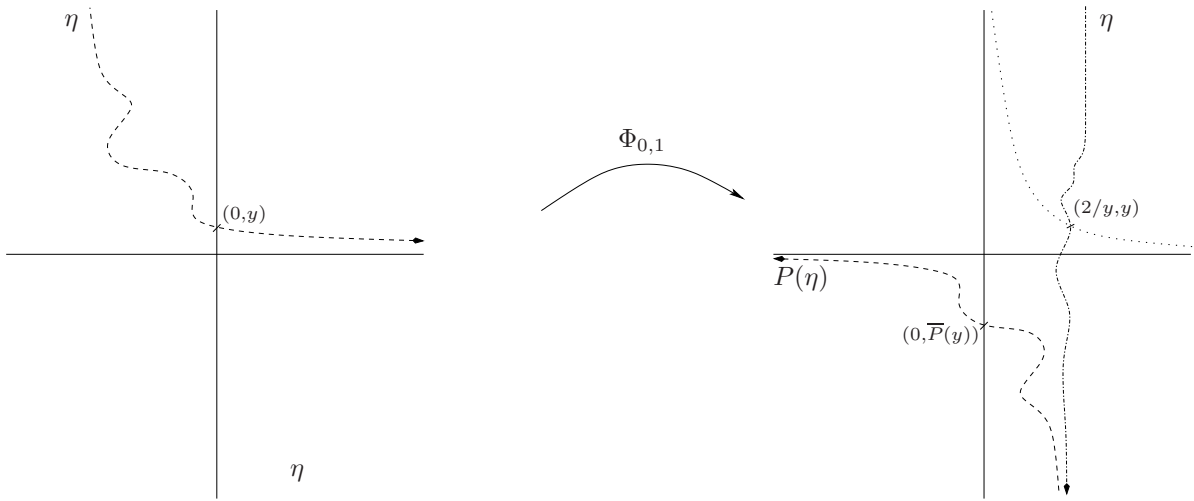}
 \caption{The map $P$ for $y>0$.}\label{fig_PPPneg}
\end{figure}

We have $\bar \eta=P(\eta)$ if and only if $\eta\cap H_1^-$ and $\bar \eta\cap H_1^-$ are asymptotic to the same vertical line. Indeed, if $\eta$ is asymptotic to $\{x=a\}$ and $\bar\eta$ to $\{x=b\}$, then for $b<a$ the curves have to intersect and so $\bar\eta\neq P(\eta)$. If $a>b$ and $\bar \eta=P(\eta)$ then the leaf of $\F_1$  asymptotic to $\{x=\frac{a+b}{2}\}$ has to cut  $\bar\eta$. But as $b<\frac{a+b}{2}$ they have to cut twice which is impossible since the foliations $\F_1$ and $\F_2$ are transverse.
It follows that $\bar \eta=P(\eta)$ if and only if $F(z)=\delta_1+\frac{2}{y_1}-h_0(y_1)$ i.e.\ that $\overline P(y)= F^{-1}(\delta_1+\frac{2}{y_1}-h_0(y_1))$ for any  $y>0$.

Let us see now what happens for a geodesic $\eta$ that intersects $H^-_0$. We denote by $(0,y_2)$ the intersection of $\eta$ with $\{x=0\}$. Note that  $y_2<0$. The curve $\eta\cap H_0^-$ is the graph of $y\mapsto \frac{2}{y}-h_0(y)+h_0(y_2)-\frac{2}{y_2}$. It is asymptotic
 to $\{x=-\delta_0+h_0(y_2)-\frac{2}{y_2}\}$ when $y$ goes to $-\infty$. Hence $P(\eta)\cap H_0$ is the graph of $y\mapsto h_0(y)-2\delta_0+h_0(y_2)-\frac{2}{y_2}$ and $P(\eta)$ cuts ${y=0}$ at the point $(-2\delta_0+h_0(y_2)-\frac{2}{y_2},0)$. 

\begin{figure}[h!]
 \labellist
 \small\hair 2pt
 \pinlabel {$\scriptstyle(0,y)$} at 140 120
 \pinlabel  {$P(\eta)$} at 260 360 
 \pinlabel  {$\eta$} at 10 170
 \pinlabel  {$P(\eta)$} at 610 360
 \pinlabel  {$\scriptstyle(0,\overline P(y))$} at 770 210
 \pinlabel{$\Phi_{0,1}$} at 447 277
 \endlabellist
 \centering \includegraphics[scale=.5]{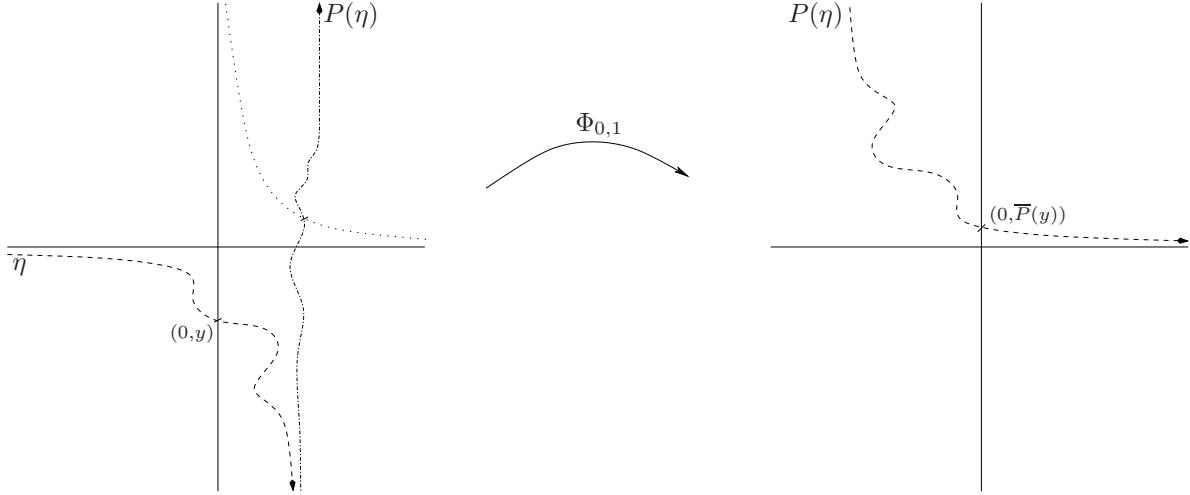}
 \caption{The map $P$ for $y<0$.}\label{fig_PPPpos}
\end{figure}

We now define the map $F$. Let $z$ be a positive number and $\bar \eta$ be the geodesic of $\F_2$ that contains the point $(0,z)$ of $H_1^+$. Its intersection with $H_0^+$ contains the point $(2/z,z)$ therefore it is the graph of the function $y\mapsto h_0(y)-h_0(z)+\frac{2}{z}$. Thus it cuts the set $\{y=0\}$ of $U_0$  at the point $(h_1(z)+\frac{2}{z},0)$. We define $G$ by $G(z)=-h_0(z)+\frac{2}{z}$. Similarly to the previous case, we have $\bar\eta=P(\eta)$ if and only if $G(z)=-2\delta_0+h_0(y_2)-\frac{2}{y_2}$ i.e.\ $\overline P(y)= G^{-1}(-2\delta_0+h_0(y_2)-\frac{2}{y_2})$ for any $y<0$.
\end{proof}

We now choose some spacelike Zoll cylinder $(C,g)$ admitting a parabolic atlas such that $\kappa_0$, $\kappa_1$ coincide with $y\mapsto y^2$ on  neighborhood of $0$ but $\delta_0\neq\delta_1$ (using Lemma \ref{lem_PPM}'s notations). Clearly, such a surface exits, but only for $k\geq 3$. Near $0$ we thus have $h_0(s)=h_1(s)=s$ and therefore for small $z$ and  large $y>0$ 
$$\begin{aligned}
   F(z)=&-\delta_1+z-\frac{2}{z}, &   G(z)=& \frac{2}{z}-z,   \\
   F^{-1}(y)=& \frac{y+\delta_1-\sqrt{(y+\delta_1)^2+8}}{2},\quad &
   G^{-1}(y)=& \frac{-y+\sqrt{y^2+8}}{2},
  \end{aligned}
  $$
where we used that $F^{-1}$ and $G^{-1}$ tend to $0$ when $y\rightarrow +\infty$. Hence, for small $y$,
\begin{equation}
 \overline P(y)=  \left\{ 
 \begin{aligned}
 \frac{1}{2}\left(\frac{2}{y}+2\delta_1-y-\sqrt{\left[\frac{2}{y}+2\delta_1-y\right]^2+8}\right) & \ \text{if}\  y>0 \\
 \frac{1}{2}\left(\frac{2}{y}+2\delta_0-y+\sqrt{\left[\frac{2}{y}+2\delta_0-y\right]^2+8}\right) & \ \text{if}\ y<0.
 \end{aligned}
 \right.
\end{equation}
Consequently $\overline P$ is $C^1$ but not $C^2$. It means that the metric is not $C^2$-conformal to a finite cover of de Sitter space and that its conformal boundary is not $C^2$. 
\end{proof}
%
\section{Elliptic case}\label{section_elliptic}
 Now we look at spacelike Zoll surface admitting a periodic Killing field. We will call them \emph{elliptic} spacelike Zoll surfaces.
This case is much simpler than the former one and very similar to the Riemannian one treated in \cite{Besse}. 
Moreover, in this case we don't need to  make any extra assumptions on the metric.
Our result in this case is the following:

\begin{thm}\label{theo_elliptic}
 If $(C,g)$ is an elliptic cylinder all of whose spacelike geodesics are closed then there exist a smooth function $f:\R\rightarrow \R$  such that $f(y)(y^2+1)-1<0$ for all $y\in \R$  and  numbers $l>0$ and $\tau>0$ such that $(C,l\,g)$ is isometric to the quotient of $\R^2$ endowed with the metric $(y^2+1)dx^2+2dxdy+f(y)dy^2$ by the translation $(x,y)\mapsto (x+\tau,y)$.
 
 Moreover such a metric has all its spacelike geodesics closed if and only if there exist $p,q$ in $\Z^*$, and an odd function $\kappa$ bounded below by $-\frac{p\tau}{2q\pi}$ such that $$f(y)=\frac{1-(\kappa(y)+\frac{p\tau}{2q\pi})^2}{y^2+1}.$$
 In particular elliptic  M\"obius strips all of whose spacelike geodesics are closed have constant curvature.
\end{thm}
%
{\bf Sketch of proof.} It can be proven simply by following the scheme of proof of Section \ref{section_parabolic}. The adaptation is straightforward.
In particular, the metric  has closed spacelike geodesics if  and only if there exists integers $p$ and $q$ such that for any $c>0$ 
$$2q\int_{-c}^c \frac{\sqrt{c^2+1}\sqrt{1-f(y)(y^2+1)}}{(1+y^2)\sqrt{c^2-y^2}}dy=p\tau,$$
but the cylinder is spacelike Zoll if and only if $p=1$. Using the fact that 
$$\int_{-c}^c\frac{\sqrt{c^2+1}}{(1+y^2)\sqrt{c^2-y^2}}dy=\pi$$
we find that the metric has closed spacelike geodesics if and only if 
\begin{equation}\label{eq_elliptic}
 \int_{-c}^c \frac{\sqrt{c^2+1}(\sqrt{1-f(y)(y^2+1)}-\frac{p\tau}{2\pi q})}{(1+y^2)\sqrt{c^2-y^2}}dy=0 \qquad \forall c>0
\end{equation}
Adapting Lemma \ref{lem_Besse} we see that it implies that $y\mapsto \sqrt{1-f(y)(y^2+1)}-\frac{p\tau}{2\pi q}$ is odd. 

 For metrics $g$ lifted from the M\"obius band this  implies that $\sqrt{1-f(y)(y^2+1)}-\frac{p\tau}{2\pi q}$ vanishes and $g$ has constant curvature.$\Box$
%
%
\section{The hyperbolic case}\label{section_hyperbolic}
We are interested now in spacelike Zoll surfaces with a Killing field that vanishes somewhere. Again we start by describing a family of Lorentzian atlases.
\begin{defi}\label{def_hyp}
Let $(S,g)$ be a Lorentzian surface and let $\mathcal A=\{(U_i,\phi_i); i\in\zkkk\}$ be an atlas of it. We denote by $\Phi_{ij}=\phi_{j}\circ\phi^{-1}_{i}$ the transition functions of $\mathcal A$.\\
We will say that $\mathcal {A}$ is a hyperbolic atlas with parameter $\tau$ of  $(S,g)$ if:
\begin{enumerate}
 \item for all $i\in\zkkk$, the image of $\phi_i$ is $\R^2$;
 \item the transition functions are the following:
 $$
 \begin{array}{lllll}
                    & \Phi_{2i,2i+1}:  & P_0  &\rightarrow  &P_0\\
                    &                  & (x,y)&\mapsto      & \left(-x+\log\left(\frac{y+1}{-y+1}\right),y\right),\\
\text{if\ }i\neq 0  & \Phi_{2i-2,2i+1}:& P_-  &\rightarrow  &P_-\\
                    &                  & (x,y)&\mapsto      & \left(-x+\log\left(\frac{y+1}{y-1}\right),y\right),\\
                    & \Phi_{4k-2,1}:   & P_-  &\rightarrow  & P_-\\
                    &                  &(x,y) &\mapsto      &\left(-x+\log\left(\frac{y+1}{y-1}\right)+\tau,y\right),\\
\text{if\ } i\neq 0 & \Phi_{2i-1,2i}:  & P_+  &\rightarrow  & P_+\\
                    &                  &(x,y) &\mapsto      &\left(-x+\log\left(\frac{y+1}{y-1}\right),y\right),\\
                    & \Phi_{4k-1,0}:   & P_+  &\rightarrow  & P_+\\
                    &                  &(x,y) &\mapsto      &\left(-x+\log\left(\frac{y+1}{y-1}\right)-\tau,y\right),\\
  \end{array}
 $$
 where $P_+=\{(x,y)\in \R^2;y>1\}$, $P_0=(x,y)\in \R^2;-1<y<1\}$ and  $P_-=\{(x,y)\in \R^2; y<-1\}$ and $\tau\in \R$;
 \item for all $i\in\{1,\dots 4k\}$, $$g_i=\phi_i^{-1}{}^*g= (y^2-1)\, dx^2+2dxdy+f_i(y)dy^2,$$
 where $f_i$ is a smooth function satisfying $1-(y^2-1)f_i(y)>0$ for all $y\in \R$.
\end{enumerate}
\end{defi}
\begin{figure}[h]\centering
 \includegraphics[scale=.33]{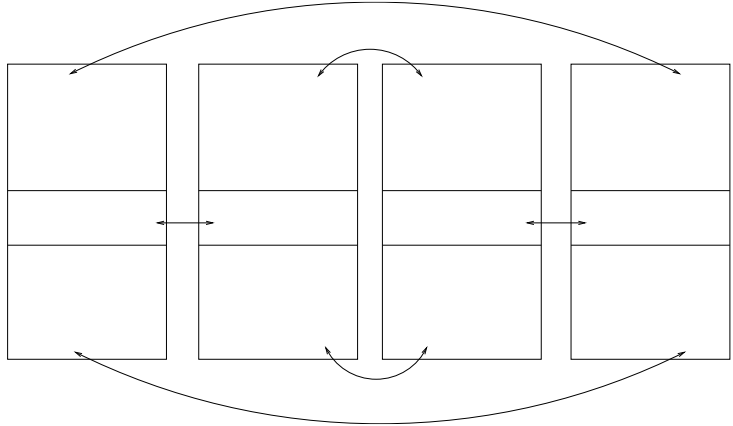}
\caption {The gluing picture when $k=1$.}
 \end{figure}
\begin{figure}[h]
\centering
 \includegraphics[scale=.4]{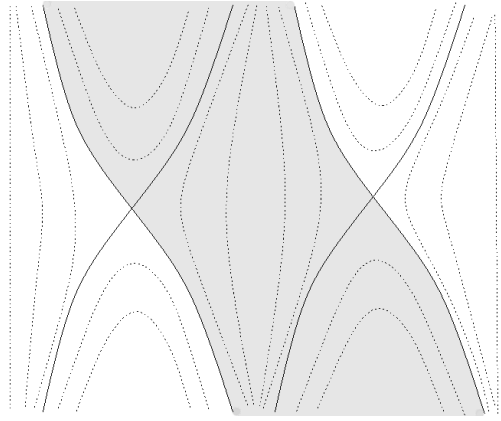} 
 \caption{The manifold, the open set $U_0$ and the Killing field  ($k=1$).}
\end{figure}
\begin{rema}
\begin{enumerate}
\item The function $\Phi_{i,j}$ are odd involutions.
\item If $(S,g)$ admits a hyperbolic atlas then $S$ is diffeomorphic to a cylinder with $2k$ points removed.
\item DeSitter space (with $2$ points removed) clearly admits such an atlas with   the parameters $k=1$, $\tau=0$ and  $f_1=\dots =f_4=0$.
 \item Note that a hyperbolic atlas induces an analytic structure on $C$. The Killing field $K$ is according to the conditions analytic as well. In opposition the 
metric $g$ need not be analytic, but the $g$-length of $K$ is again an analytic function on $C$. 
 \item The transition maps being isometries, the restriction of $f_{2i}$ and $f_{2i+1}$ coincide on $P_0$, $f_{2i-2}$ and $f_{2i+1}$  coincide on $P_-$ and $f_{2i-1}$ and $f_{2i}$ coincide on $P_+$. It follows that  it is sufficient to know the $f_{2i}$ in order to know all the $f_i$. 
 Further if $\mathcal A$ is  analytic then $f_0=f_1=\dots=f_{4k}$.
\end{enumerate}
\end{rema}
\begin{prop}\label{embedded}
 If $(S,g)$ is a Lorentzian surface admitting a hyperbolic atlas then it can be isometrically embedded into a Lorentzian cylinder.
\end{prop}
\begin{proof}  Since we fill the holes one by one, we need to consider the 
case $k=1$ only. In this case $S$ is diffeomorphic to a cylinder with 2 points removed. Let us see how to fill one hole.

Let $F:\R^2\rightarrow \R$ be a smooth function invariant by the flow of $u\partial_u-v\partial_v$ defined by 
$F(u,v)=f_0(uv-1)$ if $u>0$ and $F(u,v)=f_2(uv-1)$ if $u<0$. Let $h$ be the Lorentzian metric on $\R^2$ given by 
$h= v^2du^2+2(1+uv\,F)du\,dv+u^2F\,dv^2$.

Let $V_0=\{(u,v)\in \R^2\,;\, u>0\}$, $V_1=\{(u,v)\in \R^2\,;\, v>0\ \text{and}\ 2-uv>0\}$, $V_2=\{(u,v)\in \R^2\,;\, u<0\}$ and $V_3=\{(u,v)\in \R^2\,;\, v<0\ \text{and}\ 2-uv>0\}$. Let $\psi_0$, $\psi_1$, $\psi_2$ and $\psi_3$ the diffeomorphisms defined by 
\begin{eqnarray*}
 \psi_0: V_0 & \rightarrow & \R^2\\
 (u,v)&\mapsto& (\log(u),uv-1)\\
 \psi_1: V_1 & \rightarrow & \{(x,y)\,;\, y<1\}\\
 (u,v)&\mapsto & \left(\log\left(\frac{v}{2-uv}\right),uv-1\right)\\
 \psi_2: V_2 & \rightarrow & \R^2\\
 (u,v)&\mapsto& (\log(-u)+\tau ,uv-1)\\
 \psi_3: V_3 & \rightarrow & \{(x,y)\,;\, y<1\}\\
 (u,v)&\mapsto & \left(\log\left(\frac{-v}{2-uv}\right)-\tau ,uv-1\right) 
\end{eqnarray*}
We let the reader check that $\psi_i^*g_i=h|_{V_i}$ and that $\psi_{i+1}\circ\psi_i^{-1}=\Phi_{i,i+1}$. Hence the first hole is filled, the second one can be filled the same way.
\end{proof}

The following proposition gives sufficient conditions to ensure that a spacelike Zoll cylinder admits a hyperbolic atlas. 
The conditions \ref{cond_2} and \ref{cond_3} are necessary in the sense that they are satisfied by metrics admitting a hyperbolic atlas. There exists spacelike Zoll metrics admitting a hyperbolic atlas that do not verify condition \ref{cond_1}, see Theorem \ref{theo_hyperbolic}, but this condition is verified on a neighborhood of de Sitter space.

\begin{prop}\label{prop_hyperbolicatlas}
 Let $(C,g)$ be  a 
 spacelike Zoll cylinder with a Killing field $K$ that is timelike somewhere (and therefore vanishing somewhere) and  
 $\{p_0,\dots,p_{2k-1}\}$ be the set of zeros of $K$. Let $\eta$ be  a lightlike geodesic  transverse to $K$. 
 Then there exists $l>0$ such that  $(C\smallsetminus \{p_0,\dots,p_{2k-1}\},l\cdot g)$ admits a hyperbolic atlas if 
 \begin{enumerate}
  \item\label{cond_1} the curvature of $g$ is positive at any point where $K$ is timelike or lightlike,
  \item\label{cond_2} $K$ and $\alpha$ are analytic,
  \item\label{cond_3} there exists $t_1 \neq  t_2$ such that $\alpha\circ \eta(t_1)=\alpha\circ \eta(t_2)=0$ and  $(\alpha\circ \eta)'(t_1)=-(\alpha\circ \eta)'(t_2)$.
 \end{enumerate}
\end{prop}

\begin{proof} Let $\{(U_i,\psi_i),i\in \Z/4k\Z\}$ be the atlas of  $C\smallsetminus\{p_0,\dots,p_{2k-1}\}$ given by  Proposition \ref{prop_atlases}. In any of these charts $g$ reads $h_i(y)dx^2+2dxdy$. Condition \ref{cond_3} of the statement above can be translated in $h_0'(a_0)=-h'_0(b_0)$, with $a_0$ and $b_0$ as defined in Proposition \ref{prop_atlases}. It follows from the compatibility conditions between the $h_i$ that for any $i\in \Z/4k\Z$, we have $h'(a_i)=h'(a_0)$ and $h'(b_i)=h'(b_0)$ and therefore $h'(a_i)=-h'(b_i)$ Let $y_i$ be a critical point of $h_i$, we know that $h_i(y_i)< 0$ and it follows from our assumption on the curvature that $h_i''(y_i)>0$. Therefore the function $h_i$ has a unique minimum.  Possibly multiplying $g$ and $K$ by  positive constants,  we may assume that  the minimum of $h_0$ is $-1$ and  $|h_0'(a_0)|=|h_0'(b_0)|=2$.


 

The space of non constant orbits of $K$ is an analytic non Hausdorff $1$-dimensional manifold that we denote $\cal L$. The points where $\cal L$ is not separated correspond to the separatrix of the saddle points of $K$. The cardinal of this set is therefore $8k$. An atlas of $\cal L$ can be easily obtained from  the atlas $\{(U_i,\psi_i)\}$. To each $U_i$ corresponds a maximal connected Hausdorff submanifold  $V_i$ of $\cal L$.  We endow $\cal L$ with an  analytic vector field $\partial_s$ whose restriction to $V_0$ is complete. It gives  a coordinate $s$ on each line $D_i=\{(x,y)\in U_i, x=0\}$ such that $\alpha|_{D_i}(s)$ is analytic. It follows from the gluing picture of $\cal L$, that, up to a right translation, the functions $\alpha|_{D_i}(s)$ are all the same. In particular it means that all the functions $h_i$ have the same minimum and, using the fact that all $h_i$ goes to infinity at both ends, that $s$ goes from $-\infty$ to $+\infty$ on 
each $V_i$, i.e.\ that $\partial_s$ is complete.

Let $\lambda_i$ be the function such that $h_i(y)=e^{\lambda_i(y)}(y-a_i)(y-b_i)$. Notice that $\lambda_i(a_i)=\lambda_i(b_i)=\ln \left(\frac{2}{|a_i-b_i|}\right)$. The map $\alpha|_{D_i}$ is a Morse function admitting a unique critical point therefore there exists a coordinate $t$ on ${D_i}$ (depending on $i$) such that $\alpha|_{D_i}(t)=t^2-1$ and therefore that $y(-1)=a_i$ and $y(1)=b_i$.
Differentiating the equality $e^{\lambda_i(y(t))}(y(t)-a_i)(y(t)-b_i)=t^2-1$ we obtain 
$y'(-1)e^{\lambda_i(a_i)}(a_i-b_i)=-2$ i.e.\  $y'(-1)=1$ and similarly $y'(1)=1$.
The metric $g$ reads as $(t^2-1)dx^2 +2\beta_i(t)dtdx$ in the coordinates $(x,t)$. 
Let $t\mapsto c(t)$ be a solution of $c'(t)(t^2-1)=1-\beta_i(t)$.  It follows from  $y'(\pm 1)=1$ that $\beta_i(\pm 1)=1$, consequently 
 $c'$, and therefore $c$, is smooth.


 Using the solution $c(t)$ we define new coordinates $(u,t)$ on $U_i$ by $(u,t)\mapsto (c(t)+u,t)$. The metric in the coordinates $(u,t)$  has the desired form. As in Proposition \ref{prop_parabolicatlas}, we see that the transition functions are isometries preserving the second coordinate and sending the Killing field to its opposite. Consequently they have the desired expression up to a horizontal translation of length $\tau_i$. However in this case, it is not obvious which translations can be supposed to be trivial.

In order to conclude we remark that the restrictions of the function $\alpha$ to arc length parametrized geodesics intersecting $\{p_0,\dots,p_{2k-1}\}$ does not depend of the choice of these horizontal translations. Furthermore, it is proven in \cite{BavardMounoud} that these functions completely determine the metric on a neighborhood of the zero. Thus it follows from Proposition \ref{embedded} that  $\sum_{k=0}^3\tau_{4j+k}=0$ for any $j$. Now it is not difficult to modify the atlas in order to obtain a hyperbolic atlas.\end{proof}

\begin{thm}\label{theo_hyperbolic}
Let $(C,g)$ be a Lorentzian cylinder  and $\{p_0,\ldots, p_{2k-1}\}$ such that $(C\smallsetminus\{p_0,\ldots, p_{2k-1}\},g)$ admits a  hyperbolic atlas  $\mathcal A=\{(U_i,\phi_i); i\in \zkkk \}$. 
If the parameter $\tau$ of $\mathcal A$ is $0$ and if there exist  $2k$ smooth functions  $\kappa_0,\dots \kappa_{2k-1}$ from $\R$ to $\R$ 
such that 
 \begin{enumerate}
  \item\label{o} for all $t\in \R$, for all $j\in \zkk$, $\kappa_j(t)\geq -1$;
  \item\label{t} all the functions $\kappa_j$ have  the same infinite Taylor expansion at $-1$ and at $1$ and satisfy  $\kappa_j(\pm 1)=0$;
  \item for all $i\in \{0,\dots, 2k-1\}$ the function $f_{2i}$ such that 
    $$g_{2i}=\phi_i^{-1}{}^*g=(y^2-1)  dx^2+2dxdy+ f_{2i}(y)dy^2,$$ 
    satisfies $$f_{2i}=\frac{1-(1+\kappa_{i})^2}{y^2-1}.$$
  \item\label{th} the restriction of the function $\sum_{j} \kappa_j$ to $[-1,1]$  and the restrictions of the functions $\sum_{j} \kappa_{2j}$ and $\sum_{j} \kappa_{2j+1}$ to  $]-\infty,-1]\cup [1,+\infty[$ are odd.
  
 \end{enumerate}
then  the cylinder $(C,g)$ is spacelike Zoll.

Moreover, we have a reciprocal in the analytic case i.e.\ if $(C,g)$ is analytic and spacelike Zoll then the parameter $\tau$ is equal to $0$ and $\kappa_0=\dots=\kappa_{2k}$ are odd functions.
\end{thm}
\begin{proof}[ Proof of Theorem \ref{theo_hyperbolic}.]
Let $\mathcal A$ be a hyperbolic atlas of $(C\smallsetminus\{p_0,\ldots, p_{2k-1}\},g)$. We denote by $K$ the associated Killing field and  by $g_i$ the expression of $g$ in the coordinates $(U_i,\phi_i)$.  We recall that there exist functions $f_i$ such that the $g_i$'s read as $(y^2-1) dx^2+2dxdy+ f_i(y)dy^2$.

We first remark that on each $U_i$ the foliation perpendicular to $K$ does not depend on the functions $f_i$ and so do the $\Phi_{i,j}$. Moreover, in the proof of 
Proposition \ref{embedded} we saw that transitions functions used to fill the holes are also independent from these functions. It means that the (unparameterized) 
spacelike geodesics orthogonal to $K$ do not depend on the choice of the functions $f_i$ but only on $\tau$. In order to see when such a geodesic is simply 
closed we can assume that all the $f_i$ are $0$. The cylinder is then the quotient of the universal cover of de Sitter by the product of an elliptic element and the 
time $\tau$ of an hyperbolic flow. Consequently, the spacelike geodesics orthogonal to $K$ are simply closed if and only if $\tau=0$. We assume now $\tau=0$ 
and we study spacelike geodesics  not perpendicular to $K$.
\begin{lem}\label{lem_tangenthyperbolic}
Let $\gamma_i\,:\,t\mapsto(x(t),y(t))$ be a unit spacelike geodesic of $(\R^2,g_i)$ that is not  perpendicular to $\partial_x$. Then there exists $c>1$ such that 
$\gamma_i$ is contained between the lines $y=c$ and $y=-c$. Further $\gamma_i$ is either tangent exactly once to each of these lines or it is  asymptotic to the 
lines $y=\pm 1$. Moreover, in the first  case the geodesic segment between the points of tangency satisfies:
 $$\frac{\partial x}{\partial y}= \frac{\sqrt{c^2-1}\sqrt{1-(y^2-1)f_i(y)}-\sqrt{c^2-y^2}}{(y^2-1)\sqrt{c^2-y^2}}.$$ 
Furthermore, these two situations are exchanged by a transition map $\Phi_{2j,2j+1}$. 
\end{lem}
\begin{proof}[Proof of Lemma \ref{lem_tangenthyperbolic}]
Let $\gamma_i\,:\,t\mapsto(x(t),y(t))$ be unit spacelike geodesic of $(\R^2,g_i)$ that is  not  perpendicular to $\partial_x$.

Writing the first integrals of the geodesic flow, we have 
\begin{equation} 
 \begin{cases}
 (y^2-1)x'+y'=\epsilon_1 \sqrt{c^2-1}\\
 (y^2-1)x'^2+2x'y'+f_i(y)y'^2=1,
\end{cases}
\end{equation}
with $c>1$ and $\epsilon_1=\pm 1$.

This system of equations can be solved if and only if $c^2-y^2\geq 0$ proving that $-c\leq y\leq c$. Solving it we find:
\begin{eqnarray*}
  x'&=&\frac{\epsilon_1 \sqrt{c^2-1} (1-(y^2-1)f_i(y))+\epsilon \sqrt{(1-y^2f_i(y))(c^2-y^2)}}{(y^2-1)(1-y^2f_i(y))}\\
  y'&=&-\epsilon \sqrt{\frac{c^2-y^2}{1-(y^2-1)f_i(y)}},
  \end{eqnarray*}
  where $\epsilon =\pm 1$. It implies that 
  $$\frac{\partial x}{\partial y}= \frac{-\epsilon\epsilon_1 \sqrt{c^2-1}\sqrt{1-(y^2-1)f_i(y)}-\sqrt{c^2-y^2}}
  {(y^2-1)\sqrt{c^2-y^2}}.$$
The number $\epsilon_1$ determines the orientation of the geodesic and $\epsilon$ changes only when $y=\pm c$.

The fact that for any $y_0$ such that $1<y_0<c$ the integrals
\begin{eqnarray*}
\int_1^{y_0} \frac{- \sqrt{c^2-1}\sqrt{1-(y^2-1)f_i(y)}-\sqrt{c^2-y^2}}{(y^2-1)\sqrt{c^2-y^2}}dy\\
 \int_{-y_0}^{-1} \frac{- \sqrt{c^2-1}\sqrt{1-(y^2-1)f_i(y)}-\sqrt{c^2-y^2}}{(y^2-1)\sqrt{c^2-y^2}}dy
\end{eqnarray*}
diverge, implies that $\gamma$ is tangent at most once to each line $y=\pm c$.

If $\gamma$ intersect $P_+\cup P_-$, then the fact that for any $y_0\in ]1,c[$ and any $y_1\in ]-c,-1[$  the integrals
\begin{eqnarray*}
 \int_{-c}^{c}\frac{\sqrt{c^2-1}\sqrt{1-(y^2-1)f_i(y)}-\sqrt{c^2-y^2}}{(y^2-1)\sqrt{c^2-y^2}}dy\\
 \int_{y_0}^c \frac{- \sqrt{c^2-1}\sqrt{1-(y^2-1)f_i(y)}-\sqrt{c^2-y^2}}{(y^2-1)\sqrt{c^2-y^2}}dy\\
 \int_{-c}^{y_1}\frac{- \sqrt{c^2-1}\sqrt{1-(y^2-1)f_i(y)}-\sqrt{c^2-y^2}}{(y^2-1)\sqrt{c^2-y^2}}dy
\end{eqnarray*}
converge implies that $\gamma$ is tangent at least once to each line $y=\pm c$. Between these points we have
 $$\frac{\partial x}{\partial y}= \frac{\sqrt{c^2-1}\sqrt{1-(y^2-1)f_i(y)}-\sqrt{c^2-y^2}}{(y^2-1)\sqrt{c^2-y^2}}.$$ 

If $\gamma\subset P_0$ then  $\frac{\partial x}{\partial y}$ has to be equal to  $$\frac{-\sqrt{c^2-1}\sqrt{1-(y^2-1)f_i(y)}-\sqrt{c^2-y^2}}{(y^2-1)\sqrt{c^2-y^2}}$$ 
and the geodesic is asymptotic to the lines $y=\pm 1$. 
\end{proof}
%
\begin{prop}\label{prop_recoll_hyp}
 Let $\gamma$ be a unit spacelike geodesic of $(C,g)$ that is  not perpendicular to $K$. 
 Let $i_0$ be an \emph{even} element of $\zkkk$ such that $\gamma$ is tangent to $K$ at some point of $U_{i_0}$.
 The geodesic $\gamma$ is closed if and only if 
 \begin{equation}\label{cond_hyperbolic}
   \int_{-c}^c \sqrt{c^2-1}\frac{\sum_{i\in\sigma_\gamma} (\sqrt{1-(y^2-1)f_i(y)}-1)}{(y^2-1)\sqrt{c^2-y^2}}dy=0,
 \end{equation}
 where $\sigma_\gamma=\{2i+\frac{1+(-1)^{i+1}}{2}+i_0 \in \zkkk \}$ and $\sqrt{c^2-1}=|g(\gamma',K)|$.
\end{prop}
\begin{proof} Similar to the parabolic case, the integral above expresses the shift between the geodesic $\gamma$ and the geodesic of $g^0$ starting with the same initial speed at a point of tangency. The only difference is that when $\gamma$ is cut along its points of tangency with $K$ the segments obtained are contained in a $U_i$ with $i\in \sigma_\gamma$.
\end{proof}
\hop
Let $g$ be a metric having a hyperbolic atlas with $\tau=0$. Replacing $\sqrt{1-(y^2-1)f_{2i}(y)}-1$ by $\kappa_i$ in \eqref{cond_hyperbolic}, we see that the spacelike geodesics having a point of tangency in $U_0$ are all closed  if and only if
$$ \int_{[-c,-1]\cup[1,c]} \frac{ \sqrt{c^2-1}\sum_{i\in\sigma_\gamma} \kappa_{2i}}{(y^2-1)\sqrt{c^2-y^2}}dy +\int_{-1}^{1} \frac{ \sqrt{c^2-1}\sum_{i\in\sigma_\gamma} \kappa_{i}}{(y^2-1)\sqrt{c^2-y^2}}dy=0.
$$
We see also  that spacelike geodesics having a point of tangency in $U_2$ are all closed  if and only if
$$ \int_{[-c,-1]\cup[1,c]} \frac{ \sqrt{c^2-1}\sum_{i\in\sigma_\gamma} \kappa_{2i+1}}{(y^2-1)\sqrt{c^2-y^2}}dy +\int_{-1}^{1} \frac{ \sqrt{c^2-1}\sum_{i\in\sigma_\gamma} \kappa_{i}}{(y^2-1)\sqrt{c^2-y^2}}dy=0.
$$
It is the case, under the hypothesis of Theorem \ref{theo_hyperbolic}. The reciprocal is given by applying the following lemma to the function $\sum_{i\in \sigma_{\gamma}}\kappa_i(y)/(y^2-1)$.

\begin{lem}\label{lem_Besse_hyp}
Let $h :\R\rightarrow \R$ be a smooth function. If the function 
 $$H:c\mapsto\int_{0}^{c}\sqrt{c^2-1}\frac{h(s)}{\sqrt{c^2-s^2}}ds$$
 is constant on $]1,+\infty[$ then it is equal to $0$. If moreover $h$ is analytic then $h=0$.
\end{lem}

\begin{proof} We define the function $J(a)$ by:
$$J(a)=\int_1^a \frac{c H(c)}{\sqrt{a^2-c^2}\sqrt{c^2-1}}dc=\int_1^a\frac{c}{\sqrt{a^2-c^2}}\int_0^c \frac{h(s)}{\sqrt{c^2-s^2}}ds\;dc$$
Doing the same computation as in the proof of Lemma \ref{lem_Besse} we find
\begin{equation}\label{eqeq}
J(a)=\pi\int_0^ah(s)ds - 2\int_0^1h(s)\arctan\left(\frac{1-s^2}{a^2-1}\right)ds.
\end{equation}
On the other hand if $H$ is constant equal to $\tau$ then $J(a)$ does not depend on $a$, in fact $J(a)=\tau\pi/2$.
When $a$ tends to $1$ then \eqref{eqeq} tends to $0$ therefore $\tau=0$. Proving the first part.

It follows that $\int_{0}^{c}\frac{c\ h(s)}{\sqrt{c^2-s^2}}ds=0$ for any $c>1$. If $h$ is analytic then this equality is in fact true for all $c>0$ and it follows from Lemma \ref{lem_Besse} that $h=0$.
\end{proof}
 Combining Lemma \ref{lem_tangenthyperbolic}, Proposition \ref{prop_recoll_hyp} and Lemma \ref{lem_Besse_hyp} finishes the proof of Theorem
\ref{theo_hyperbolic}.
\end{proof}

It is not clear, whereas Lemma \ref{lem_Besse_hyp} can be extended to the smooth case. Indeed, it is possible to construct non zero $C^n$ functions $h$ such that the corresponding function $J$  vanishes, even though this does not
imply that $H$ also vanishes.
Adapting the construction of Corollary \ref{coro_parabolic}, we obtain:
\begin{cor}\label{coro_mob_hyp}
 There exists smooth spacelike Zoll M\"obius strip with non constant curvature whose orientation cover admits a hyperbolic atlas but no analytic one.
\end{cor}
\begin{proof} Let $\mathcal A$ be a hyperbolic atlas with $k=2$ and $\tau=0$. Let $\sigma:C\rightarrow C$  be the involution such that $\sigma(U_{2i})=U_{2i+5}$ 
(and therefore $\sigma(U_{2i+1})=U_{2i-4}$) and $\phi_{2i+5}\circ\sigma\circ\phi_{2i}^{-1}(x,y)=-(x,y)=\phi_{2i}\circ\sigma\circ\phi_{2i+5}^{-1}(x,y)$. 
We let  the reader check that $\sigma$ is well defined. It has no fixed points and it is not orientation preserving therefore $C/\sigma$ is a M\"obius strip. Let 
$\kappa$ be a smooth function with support in $[2,3]$ and values in $[-1,1]$.
We define  four functions by setting $\kappa_0=\kappa$, $\kappa_1(t)=-\kappa(t)+\kappa(-t)$, $\kappa_2(t)=-\kappa(-t)$ and  $\kappa_3=0$.
Let $g$ be the spacelike Zoll metric provided by Theorem \ref{theo_hyperbolic}  with the functions $\kappa_i$. This metric is clearly invariant by $\sigma$ and therefore induces a  metric all of whose spacelike geodesics are closed  on $C/\sigma$. 
\end{proof}

\section{Blaschke's examples}\label{sec_Blaschke}
It is also possible to produce examples with no global Killing field. We just adapt Blaschke construction from \cite{Blaschke} to the Lorentzian case. We  give only  one of the possible constructions and let the reader imagine all the possible variations around it. 

We start with de Sitter space seen as $\{(x,y,z)\in \R^3, -x^2+y^2+z^2=1 \}$ endowed with the metric $g^0$  induced by $-dx^2+dy^2+dz^2$. Let $K_1$ be the elliptic Killing field given by $K_1(x,y,z)=(0,-z,y)$ and $K_2$ be the parabolic Killing field given by $K_2(x,y,z)=(y,x+z,-y)$. Let $V_1=\{(x,y,z)\in \R^3; -x^2+y^2+z^2=1 \text{\ and\ } g^0(K_1,K_1)\leq 2\}$. We see that $(x,y,z)\in V_1$ if and only if $|x|\leq 1$, therefore for any  $(x,y,z)\in V_1$ we have $g^0(K_2,K_2)=(x+z)^2\leq 9$. Hence, $V_2=\{(x,y,z)\in \R^3; -x^2+y^2+z^2=1 \text{\ and\ } 16 \leq g^0(K_2,K_2)\leq 25 \}$ and $V_1$ are disjoint.

Let $\kappa_1$ be an odd function with support in $[-1,1]$ bounded below by $-1$ and $g^1$ be the metric given by $g^1=(v^2+1)du^2+2dudv + \frac{1-(\kappa_1(v)+1)^2}{v^2+1}dv^2$. According to Theorem \ref{theo_elliptic}, $g^1$ induces a  spacelike Zoll metric on the quotient of $\R^2$ by the horizontal  translation of length $2\pi$ (then $p=q=1$). It can be seen as a perturbation of the de Sitter metric for which $K_1$ is still a Killing field. The support of this deformation being contained in $V_1$.

Let $\kappa_2$ be an odd function with support in $[-5,-4]\cup [4,5]$ bounded below by $-1$.  Let $g^2$ be the metric on the cylinder given by a parabolic atlas such that $k=1$, $\tau=0$ and $g^2_0=v^2du^2+2dudv + \frac{1-(\kappa_2(v)+1)^2}{v^2}dv^2$. According to Theorem \ref{theo_parabolic}, $g^2$ is spacelike Zoll. It can be seen as a perturbation of the de Sitter metric for which $K_2$ is still a Killing field, the support of this deformation being contained in $V_2$.

Let $g$ be the metric on the cylinder that coincides with $g^1$ on $V_1$, with $g^2$ on $V_2$ and with $g^0$ elsewhere.
Let $\gamma$ be a spacelike geodesic of $g$. It follows from Proposition \ref{P2} that $\gamma$ intersects $V_1$. If $\gamma$ does not meet $V_2$ then it is clearly closed, otherwise 
it has to cross it. Let  $\gamma_1$ be the $g^1$ geodesic that contains $\gamma\cap V^1$ and let $t_0<t_1<t_2<t_3$ be such that $\gamma_1([t_1,t_2])\subset V_1$, $\gamma_1(]t_0,t_1[)\cap V_1=\gamma_1(]t_2,t_3[)\cap V_1=\emptyset$ and $\gamma_1'(t_0)$ and $\gamma_1'(t_3)$ are proportional to $K_1$. The restrictions of $\gamma_1$ to $[t_0,t_1]$ and $[t_2,t_3]$ are geodesic segments  of $g^0$. 
Let us see that these two segments are in fact on the same $g^0$-geodesic. We proved that $g^1$ is spacelike Zoll by comparing its geodesics to the $g^0$ one.
The fact that for any $c>0$ we have (compare to \eqref{eq_elliptic} in section \ref{section_elliptic}):
$$
\int_{-c}^c \frac{\sqrt{c^2+1}\,\kappa_1(v)}{(1+v^2)\sqrt{c^2-v^2}}dv=0
$$
says precisely that the $g^0$-geodesic that starts tangentially to $K_1$ from $\gamma_1(t_0)$ arrives tangentially to $K_1$ at the point $\gamma_1(t_3)$, 
proving our claim. It means that seen form $V_2$ the perturbation on $V_1$ as no effect on the spacelike geodesics. In particular, in this case also $\gamma$ is 
closed and therefore $g$ is spacelike Zoll. Moreover, for $i=1,2$, any global Killing field $K$ of $g$ has to be proportional to $K_i$ on $V_i$, therefore if $K$ is 
non trivial it has both lightlike leaves and periodic leaves. But such a behaviour contradicts Proposition \ref{prop_atlases}, therefore $K=0$.

\bigskip

\begin{thebibliography}{10}
\bibitem{BavardMounoud} Ch.\ Bavard, P.\ Mounoud, \emph{Extensions de surfaces lorentziennes compactes}, in preparation.
\bibitem{Beem-Ehrlich} J.K.\ Beem; P.E.\ Ehrlich; K.L.\ Easley, \emph{Global Lorentzian geometry. Second edition.} 
Monographs and Textbooks in Pure and Applied Mathematics, 202. Marcel Dekker, Inc., New York, 1996. xiv+635 pp.
\bibitem{Besse} A.L.\ Besse, \emph{Manifolds all of whose geodesics are closed.}  Ergebnisse der Mathematik und ihrer Grenzgebiete, 93. Springer-Verlag, Berlin-New York, 1978. 
\bibitem{Blaschke} W.\ Blaschke, \emph{Vorlesungen \"uber Differentialgeometrie,} Vol.\ I, Springer, Berlin, 1924.
\bibitem{Bou}  M.\,Boucetta, \emph{M\'etriques de Lorentz Zoll sur la surface de de Sitter,}
Proyecciones 17 (1998), no. 1, 13--21. 
\bibitem{FHS} J.L. Flores, J. Herrera, M. S{\'a}nchez, \emph{On the final definition of the causal boundary and its relation with the conformal 
boundary,} Adv.\ Theor.\ Math.\ Phys. 15 (2011), no.\ 4, 991--1057.
\bibitem{Guillemin} V.\ Guillemin, \emph{The Radon transform on Zoll surfaces,} Advances in Math.\ 22 (1976), no.\ 1, 85--119.
\bibitem{Green} L.W.\ Green, \emph{Proof of Blaschke's sphere conjecture,} Bull. Amer. Math. Soc. 67 (1961) 156--158.
\bibitem{MounoudSuhr} P.\ Mounoud; S.\ Suhr, \emph{Pseudo-Riemannian geodesic foliations by circles,} Math.\ Z.\  274 (2013), no. 1-2, 225--238. 
\bibitem{penrose} R. Penrose. \emph{Techniques of differential topology in relativity}, Conference Board of the Mathematical Sciences Regional
              Conference Series in Applied Mathematics, No. 7, Philadelphia: SIAM, 1972.
\bibitem{Pries} Ch.\ Pries, \emph{Geodesics closed on the projective plane,}
Geom.\ Funct.\ Anal.\ 18 (2009), no.\ 5, 1774--1785.
\end{thebibliography}
\end{document}